\documentclass[12pt]{article}
\usepackage{tikz}
\usepackage{latexsym, graphicx, amsmath, amssymb, amsbsy, amsfonts, amsthm}
\usepackage{hyperref}
\usepackage{enumerate}
\usepackage{fullpage}

\theoremstyle{plain}
\newtheorem{lemma}{Lemma} 
\newtheorem{theorem}[lemma]{Theorem}
\newtheorem{proposition}[lemma]{Proposition}

\newtheorem{claim}{Claim}[lemma]

\theoremstyle{definition}
\newtheorem{definition}[lemma]{Definition}

\newcounter{casenum}	
\newcounter{subcasenum}	
\numberwithin{subcasenum}{casenum}
\newcounter{subsubcasenum}	
\numberwithin{subsubcasenum}{subcasenum}

\renewcommand{\thecasenum}{\arabic{casenum}}
\renewcommand{\thesubcasenum}{\thecasenum.\roman{subcasenum}}

\newcounter{stagenum}

\newenvironment{mycases}
{
  \list{}{%
    \leftmargin0.5cm   
    \rightmargin0cm
  }
  \item\relax
	\setcounter{casenum}{0}
}
{	
	\endlist
}
\newenvironment{subcases}
{
  \list{}{%
    \leftmargin0.5cm   
    \rightmargin0cm
  }
  \item\relax
}
{	
	\endlist
}

\newcommand{\mycase}[1]{
	\vspace{0.5em}
	
	\refstepcounter{casenum}
	\noindent\hspace{-0.5cm}\textit{Case \thecasenum: #1} 
}

\newcommand{\subcase}[1]{
	\vspace{0.25em}
	
	\refstepcounter{subcasenum}
	\noindent\hspace{-0.5cm}\textit{Case \thesubcasenum: #1} 
}

\newcommand{\subsubcase}[0]{
	\refstepcounter{subsubcasenum}
	\noindent\hspace{-0.6cm}{(\alph{subsubcasenum})}
}

\begin{document}

\title{On Choosability with Separation of Planar Graphs with Forbidden Cycles}
\author{Ilkyoo Choi\thanks{Department of Mathematics, University of Illinois. \texttt{$\{$ichoi4,lidicky,stolee$\}$@illinois.edu}} 
	\and Bernard Lidick\'y$^*$
	\and Derrick Stolee$^*$}
\date{\today}
\maketitle

\vspace{-2pc}

\begin{abstract}
We study choosability with separation which is a constrained version of list coloring of graphs.
A \emph{$(k,d)$-list assignment} $L$ of a graph $G$ is a function that assigns to each vertex $v$ a list $L(v)$ of at least $k$ colors and for any adjacent pair $xy$, the lists $L(x)$ and $L(y)$ share at most $d$ colors. A graph $G$ is $(k,d)$-choosable if there exists an $L$-coloring of $G$ for every $(k,d)$-list assignment $L$. 
This concept is also known as choosability with separation.
We prove that planar graphs without 4-cycles are $(3,1)$-choosable and that planar graphs 
without 5-cycles and 6-cycles are $(3,1)$-choosable. In addition, we give an alternative 
and slightly stronger proof that triangle-free planar graphs are $(3,1)$-choosable.
\end{abstract}

\section{Introduction}

Given a graph $G$, a {\it list assignment} $L$ is a function on $V(G)$ that assigns to each vertex $v$ a list $L(v)$ of (\emph{available}) colors.
An \emph{$L$-coloring} is a vertex coloring $\varphi$ such that $\varphi(v) \in L(v)$ for each vertex $v$ and $\varphi(x)\neq \varphi(y)$ for each edge $xy$.
A graph $G$ is said to be $k$-choosable if there is an $L$-coloring for each list assignment $L$ where $|L(v)|\geq k$ for each vertex $v$. 
The minimum such $k$ is known as the choosability of $G$, denoted $\chi_\ell(G)$.
A graph $G$ is said to be $(k, d)$-choosable if there is an $L$-coloring for each list assignment $L$ where $|L(v)|\geq k$ for each vertex $v$ and $|L(x)\cap L(y)|\leq d$ for each edge $xy$. 

This concept is known as choosability with separation, since the second parameter may force the lists on adjacent vertices to be somewhat separated. 
If $G$ is $(k,d)$-choosable, then $G$ is also $(k',d')$-choosable for all $k' \geq k$ and $d' \leq d$.
A graph is $(k,k)$-choosable if and only if it is $k$-choosable.
Clearly, all graphs are $(k,0)$-choosable for $k\geq 1$. 
Thus, for a graph $G$ and each $1 \leq k < \chi_\ell(G)$, there is some threshold $d \in \{0,\dots,k-1\}$ such that $G$ is $(k,d)$-choosable but not $(k,d+1)$-choosable.

This concept of choosability with separation was introduced by Kratochv\'{\i}l, Tuza, and Voigt~\cite{ktv}.
They used the following, more general definition.
A graph $G$ is \emph{$(p,q,r)$-choosable}, if for every list assignment $L$ with $|L(v)| \geq p$ for each $v \in V(G)$
and $|L(u) \cap L(v)| \leq p-r$ whenever $u,v$ are adjacent vertices, $G$ is $q$-tuple $L$-colorable.
Since we consider only $q = 1$, we use a simpler notation.
They investigate this concept for both complete graphs and sparse graphs. 
The study of dense graphs were extended to complete bipartite graphs and multipartite graphs by F\"uredi, Kostochka, and Kumbhat~\cite{zkk1,zkkarxiv}.

Thomassen~\cite{thomassen1994} proved that planar graphs are 5-choosable, and hence they are $(5,d)$-choosable for all $d$.
Voigt~\cite{95Voigt} constructed a non-$4$-choosable planar graph, and there are also examples of non-$(4,3)$-choosable planar graphs.
Kratochv\'{\i}l, Tuza, and Voigt~\cite{ktv} showed that all planar graphs are $(4,1)$-choosable.
The question of whether all planar graphs are $(4, 2)$-choosable or not was raised in the same paper and it still remains open. 

Voigt~\cite{93Voigt} also constructed a non-$3$-choosable triangle-free planar graph.
\v Skrekovski~\cite{riste} observed that there are examples of triangle-free planar
graphs that are not $(3,2)$-choosable, and posed the question of whether or not every planar
graph is $(3,1)$-choosable;
Kratochv\'{\i}l, Tuza and Voigt~\cite{ktv} proved the following partial case of this question: 
\begin{theorem}\label{c3free}
Every triangle-free planar graph is $(3, 1)$-choosable.
\end{theorem}

We strengthen Theorem~\ref{c3free} by showing an alternative proof that uses a method developed by Thomassen; we also use this method to prove Theorem~\ref{c4free} below. 
Our inspiration was Thomassen's proof~\cite{thomassen1995-34} that every planar graph of girth 5 is 3-choosable.
We also prove the following two different partial cases:

\begin{theorem}\label{c4free}
Every planar graph without $4$-cycles is $(3, 1)$-choosable.
\end{theorem}

\begin{theorem}\label{c56free}
Every planar graph without $5$-cycles and $6$-cycles is $(3, 1)$-choosable.
\end{theorem}

These results are similar in nature to other results on the choosability of planar graphs when certain cycles are forbidden 
(see a survey of Borodin~\cite{borodin13}).
One of the motivations is Steinberg's Conjecture that states that all planar graphs containing no $4$- or $5$-cycles are 3-colorable~\cite{Steinberg}.
We construct a planar graph without cycles of length $4$ and $5$ that is not $(3, 2)$-choosable, to show that Steinberg's Conjecture cannot be extended to $(3,2)$-choosability.

Theorems~\ref{c3free} and \ref{c4free} are shown in Sections $2$ and $3$, respectively.
Theorem~\ref{c56free} uses a discharging technique, and is showed in Section $4$.

\subsection{Preliminaries and Notation}

Always $L$ is a list assignment on the vertices of a graph $G$.
In our proofs of Theorems~\ref{c3free} and \ref{c4free}, we use list assignments where vertices can have lists of different sizes.
A \emph{$(*,1)$-list assignment} is a list assignment $L$ where $|L(v)|\geq 1$ and $|L(u) \cap L(v)| \leq 1$ for every pair of adjacent vertices $u, v$.
A vertex $v$ is an \emph{$Ld$-vertex} when $|L(v)| = d$.

\newcommand{\cin}{\mathop{\mathrm{Int}}}
\newcommand{\cex}{\mathop{\mathrm{Ext}}}
Given a graph $G$ and a cycle $K\subset G$, an edge $uv$ of $G$ is a \emph{chord} of $K$ if $u,v \in V(K)$, but $uv$ is not an edge of $K$.
For an integer $k\ge 2$, a path $v_0v_1\ldots v_k$ is a \emph{$k$-chord\/} if $v_0,v_k\in V(K)$ and $v_1, \ldots, v_{k-1}\not\in V(K)$.  
If $G$ is a plane graph, then let $\cin_K(G)$ be the subgraph of $G$ consisting of the vertices and edges drawn inside the closed disc bounded by $K$,
and let $\cex_K(G)$ be the subgraph of $G$ obtained by removing all vertices and edges drawn inside the open disc bounded by $K$.
In particular, $K = \cin_K(G) \cap \cex_K(G)$. 

Note that each $k$-chord of $K$ belongs to exactly one of $\cin_K(G)$ or $\cex_K(G)$.
If the cycle $K$ is the outer face of $G$ and $Q$ is a $k$-chord of $K$, then let $C_1$ and $C_2$ be the two cycles in $K\cup Q$ that contain $Q$. Then the subgraphs $G_1=\cin_{C_1}(G)$ and $G_2=\cin_{C_2}(G)$ are the {\em $Q$-components} of~$G$.

A graph $G$ is \emph{$H$-free} if it does not contain a copy of $H$ as a subgraph.

\section{Forbidding 3-cycles}

In this section, we prove Theorem~\ref{c3free} as a corollary of the following theorem.
Observe that any $(3,1)$-list assignment on a triangle-free plane graph satisfies the conditions of the following theorem.

\begin{theorem}\label{thm:thomassen3cycle}
Let $G$ be a triangle-free plane graph with outer face $F$ with a subpath $P \subset F$
containing at most two vertices, and let $L$ be a $(*,1)$-list assignment such that the following conditions are satisfied:
\begin{enumerate}[(i)]
\item $|L(v)| \geq 3$ for $v \in V(G)\setminus V(F)$,
\item $|L(v)| \geq 2$ for $v \in V(F)\setminus V(P)$,
\item $|L(v)| = 1 $ for $v \in V(P)$,
\item no two vertices with lists of size two are adjacent in $G$,
\item the subgraph induced by $V(P)$ is $L$ colorable.
\end{enumerate}
Then $G$ is $L$-colorable.
\end{theorem}

\begin{proof}
Let $G$ be a counterexample where $|V(G)| + |E(G)|$ is as small as possible.
By the minimality of $G$, we assume that $|L(u) \cap L(v)| = 1$ for every edge $uv \in E(G)\setminus E(P)$.
If otherwise, then we can remove the edge $uv$ to obtain an $L$-coloring of $G - uv$, which is also an $L$-coloring of $G$.
It is also clear that $G$ is connected. 

We quickly prove that $G$ is 2-connected.
Suppose $v$ is a cut-vertex of $G$.
There exist nontrivial connected induced subgraphs $G_1$ and $G_2$ of $G$ such that $G_1 \cup G_2 = G$ and $V(G_1) \cap V(G_2) = \{v\}$. 
Assume by symmetry that $P \subseteq G_1$.
By the minimality of $G$, there exists an $L$-coloring $\varphi$ of $G_1$.
Let $L'$ be the list assignment on $V(G_2)$  where $L'(u) = L(u)$ if $u\neq v$ and $L'(v) = \{\varphi(v)\}$; the lists $L'$ satisfy the hypothesis on $G_2$.
By the minimality of $G$, the graph $G_2$ has an $L'$-coloring $\psi$ where $\psi(v) = \varphi(v)$, so $\varphi$ and $\psi$ form an $L$-coloring of $G$.

Since $G$ is 2-connected, the outer face is bounded by a cycle.
In the following claims, we prove that the cycle on $F$ does not have chords or certain types of 2-chords.

\begin{claim}
$F$ does not contain any chords.
\end{claim}
\begin{proof}
Suppose for the sake of contradiction that $Q=uv$ is a chord of $F$.
Let $G_1$ and $G_2$ be the two $Q$-components of $G$.
Assume by symmetry that $P \subseteq G_1$. By the minimality of $G$, there exists an $L$-coloring $\varphi$ of $G_1$.
Let $L'$ be the list assignment on $V(G_2)$ where for $x \in V(G_2)$, $L'(x) = \varphi(x)$ if $x \in \{u,v\}$ and $L'(x) = L(x)$ otherwise.
By the minimality of $G$, there exists an $L'$-coloring $\psi$ of $G_2$ with $\psi(u) = \varphi(u)$ and $\psi(v) = \varphi(v)$; together $\psi$ and $\varphi$ form an $L$-coloring of $G$.
\end{proof}

A $2$-chord $v_0v_1v_2$ of $F$ is \emph{bad} if $v_0$ or $v_2$ is an $L2$-vertex.
An $L3$-vertex $x \in V(F)$ is \emph{good} if there is no bad $2$-chord of $F$ containing $x$.

\begin{claim}
$G$ has a good vertex.
\end{claim}
\begin{proof}
Suppose that $F$ has no good vertex, so all $L3$-vertices in $F$ are contained in a bad chord.
Since $G$ is 2-connected and triangle-free, $|V(F)| \geq 4$. 
Hence $F$ contains at least one $L3$-vertex.
Among all $L3$-vertices in $F$, let $v_0$ be an $L3$-vertex with a bad 2-chord $Q=v_0v_1v_2$ such that the size of the $Q$-component $G_2$ not containing $P$ is minimized. 

Let $u$ be the neighbor of $v_2$ on $F$ that is in $G_2$.
Since $G$ is triangle-free, the vertices $u$ and $v_0$ are distinct.
Since $Q$ is a bad 2-chord, $v_2$ is an $L2$-vertex and hence $u$ is an $L3$-vertex.
Since $F$ has no good $L3$-vertices, there is a bad 2-chord $Q' = uu_1u_2$ of $F$ where
$u_2$ is an L2-vertex. 
Since $G$ is triangle-free, $u_1 \neq v_1$. 
Therefore, $Q'$ is contained in $G_2$ and the $Q'$-component not containing $P$ is properly contained within $G_2$, contradicting our extremal choice.
\end{proof}

Let $v_0v_1v_2$ be a path in $F$ where $v_1$ is a good vertex. 
There exists a color $c$ in $L(v_1)$ that does not appear in $L(v_0) \cup L(v_2)$.
We will color $v_1$ with $c$ and extend that coloring to $G - v_1$.
Let $G' = G - v_1$, and let $L' $ be the list assignment on $V(G')$ where $L'(u) = L(u)\setminus\{c\}$ for vertices $u$ adjacent to $v_1$ in $G$, and $L'(u) = L(u)$ otherwise.

The neighbors of $v_1$ are $L'2$-vertices in $G'$, and we verify that $G'$ satisfies our hypotheses.
Since $G$ is triangle-free, the neighbors of $v_1$ form an indepenent set.
Since $v_1$ is a good vertex, the $L'2$-vertices in $G'$ form an independent set. 
By minimality of $G$, the graph $G'$ has an $L'$-coloring $\varphi$.
This $L'$-coloring $\varphi$ extends to an $L$-coloring of $G$ by assigning $\varphi(v_1) = c$.
\end{proof}


\section{Forbidding 4-cycles}

In this section, we prove Theorem~\ref{c4free} using a strengthened hypothesis.
Observe that any $(3,1)$-list assignment on a $C_4$-free planar graph satisfies the conditions of the following theorem.

\begin{theorem}\label{thm:thomassen4cycle}
Let $G$ be a $C_4$-free plane graph with outer face $F$ with a subpath $P$ of $F$ containing at most three vertices, and let $L$ be a $(*,1)$-list assignment such that the following conditions are satisfied:
\begin{enumerate}[(i)]
\item $|L(v)| \geq 3$ for $v \in V(G)\setminus V(F)$,
\item $|L(v)| \geq 2$ for $v \in V(F)\setminus V(P)$,
\item $|L(v)| = 1 $ for $v \in V(P)$,
\item no two $L2$-vertices are adjacent in $G$,
\item the subgraph induced by $V(P)$ is $L$ colorable,
\item no vertex with list of size two is adjacent to two vertices of $P$.
\end{enumerate}
Then $G$ is $L$-colorable.
\end{theorem}

\begin{proof}
Let $G$ be a counterexample where $|V(G)| + |E(G)|$ is as small as possible.
Moreover, we assume that the sum of the sizes of the lists is also as small as possible
subject to the previous condition.
By the minimality of $G$, we assume that for every edge $uv \in E(G)\setminus E(P)$, $|L(u) \cap L(v)| = 1$.
If otherwise, then we can remove the edge $uv$ to obtain an $L$-coloring of $G - uv$, which is also an $L$-coloring of $G$.
It is also clear that $G$ is connected.

Moreover, we show $G$ is 2-connected. 
Suppose $v$ is a cut-vertex of $G$.
There exist nontrivial connected induced subgraphs $G_1$ and $G_2$ such that $G_1 \cup G_2 = G$ and $V(G_1) \cap V(G_2) = \{v\}$. 
Suppose $P$ is contained within exactly one of $G_1$ or $G_2$; by symmetry $P \subseteq G_1$.
By the minimality of $G$, there exists an $L$-coloring $\varphi$ of $G_1$.
Let $L'$ be the list assignment on $V(G_2)$  where $L'(u) = L(u)$ if $u\neq v$ and $L'(v) = \{\varphi(v)\}$.
By the minimality of $G$, there exists an $L'$-coloring of $G_2$ and this coloring combined with $\varphi$ gives an $L$-coloring of $G$.
When $P$ is not contained within only one of $G_1$ or $G_2$, we have $v \in V(P)$.
By the minimality of $G$, both $G_1$ and $G_2$ are $L$-colorable and these colorings agree on $v$ which gives an $L$-coloring of $G$.

In Claims~\ref{firstclaimc4} through \ref{lastclaimc4}, we determine certain structural properties of our counterexample $G$.
A vertex $v$ is a \emph{middle vertex} if it has degree two in $P$.
Observe that since $G$ is $C_4$-free, any two vertices have at most one common neighbor.

\begin{claim}\label{cl:septriangle}
\label{firstclaimc4}
$G$ does not contain 
a triangle with nonempty interior.
\end{claim}

\begin{proof}
Suppose not and assume $T=pqr$ is a 
triangle with nonempty interior in $G$.
Let $G_1= \cex_T(G)$ and $G_2 = \cin_T(G)$. 
Since $T$ has nonempty interior, $|V(G_1)| < |V(G)|$, and there exists an $L$-coloring $\varphi$ of $G_1$.
Let $G'$ be obtained from $G_2$ by removing the edge $rp$ and let $L'$ be a list assignment on $V(G')$ where $L'(v) = \{\varphi(v)\}$ if $v \in \{p,q,r\}$ and $L'(v) = L(v)$ otherwise.
The hypothesis applies to $G'$ and $L'$ with $pqr$ as the path on three precolored vertices on the outer face of $G'$.
Since $|E(G')| <  |E(G)|$, there exists an $L'$-coloring $\psi$ of $G'$ which combined with $\varphi$ forms an $L$-coloring of $G$.
This contradicts $G$ being a counterexample.
\end{proof}

If $|V(F)| \leq 4$, then $|V(F)| = 3$ since $G$ contains no 4-cycles.
By Claim~\ref{cl:septriangle}, $G=F$ and it is easy to check 
Theorem~\ref{thm:thomassen4cycle} for graphs with at most three vertices.
Thus, $|V(F)| \geq 5$.

A chord $Q=uv$ is \emph{bad} if one of the $Q$-components is a triangle $uvx$ where $|L(x)|=2$.
Otherwise, the chord $Q$ is \emph{good}.

\begin{claim}\label{claim:chord}\label{claim:chordnospecial}
$F$ contains only bad chords.
\end{claim}

\begin{proof}
For a good chord $Q=uv$, let $G_1$ and $G_2$ be the $Q$-components such that $|V(G_1)\cap V(P)| \geq |V(G_2)\cap V(P)|$.
If $F$ contains a good chord, select a good chord $Q$ that minimizes $|V(G_2)|$.
Since $Q$ is good, the vertices $u$ and $v$ are at distance at least three apart
in the path $F \cap G_2$.
Assume $v \notin V(P)$.

By the minimality of $G$, there exists an $L$-coloring $\varphi$ of $G_1$.
Let $L'$ be the list assignment on $V(G_2)$ where $L'(x) = \{\varphi(x)\}$ if $x\in\{u,v\}$ and $L'(x) = L(x)$ otherwise.
Since $uv$ is a chord, 
Since $G_2$ contains fewer vertices of $P$ than $G_1$, the graph $G_2$ has at most three $L'1$-vertices, and they form a path of length at most two on the outer face of $G_2$.

Since we only changed the lists on $u$ and $v$ in $G_2$, the $L'2$-vertices remain an independent set.
The only condition that remains to be verified is that every $L'2$-vertex in $G_2$ has at most one $L'1$-neighbor.

Suppose there exists an $L'2$-vertex $x \in V(G_2)$ adjacent to two $L'1$-vertices.
Since $x$ is not adjacent to two $L2$-vertices, one of these vertices must be $v$, which is not an $L1$-vertex.
Since $G$ has no 4-cycles, these two $L'1$-vertices must be adjacent, so $x$ is adjacent to $u$ and $v$.
Since $|L(x)|=2$, if either $ux$ or $vx$ is a chord, then it must be a good chord, so this contradicts the choice of $Q$. 
Hence both $ux$ and $vx$ are edges of $F$.
Moreover, Claim~\ref{cl:septriangle} implies that $G_2$ is exactly the triangle $uvx$,
which contradicts that $Q$ is a good chord.

Hence there exists an $L'$-coloring $\psi$ of $G_2$ that agrees with $\varphi$ on $Q$, and these  colorings together form an $L$-coloring of $G$.
\end{proof}

\begin{claim}\label{claim:chordP}
$F$ contains only bad chords $uv$ where $u,v \not\in V(P)$.
\end{claim}
\begin{proof}
Suppose for a contradiction that $uv$ is a bad chord and $u \in V(P)$.
Let $z \in V(F)$ be a common neighbor of $u$ and $v$ forming the bad chord.
Since $|L(u)| = 1$, $L(u) \subset L(v)$ and  $L(u) \subset L(z)$.
Hence $L(v) \cap L(z) = L(u)$. 
By the minimality of $G$, there exists an $L$-coloring of $G-vz$.
However, it is also an $L$-coloring of $G$.
\end{proof}

A 2-chord $Q=v_0v_1v_2$ of a cycle $K$ is \emph{separating} if $v_0v_2 \not\in E(K)$.
We now eliminate the possibility of $F$ containing certain separating $2$-chords.

\begin{claim}\label{claim:2chord}
$F$ does not contain a separating $2$-chord $v_0v_1v_2$ where $|L(v_2)| = 2$ and $v_0$ is not a middle vertex.
\end{claim}
\begin{proof}
For a separating 2-chord $Q = v_0v_1v_2$ where $v_2$ is an $L2$-vertex and $v_0$ is not a middle vertex, let $G_1$ and $G_2$ be the $Q$-components of $G$ where $G_1$ contains the vertices of $P$.
If such a 2-chord exists, select $Q$ to minimize $|V(G_2)|$.

By the minimality of $G$, there exists an $L$-coloring $\varphi$ of $G_1$.
Let $L'$ be the list assignment on $G_2$ where $L'(v_i) = \{\varphi(v_i)\}$ for $i \in \{0,1,2\}$ and $L'(x) = L(x)$ for $x \in V(G_2) \setminus V(Q)$.
The $L'1$-vertices of $G_2$ are exactly $v_0$, $v_1$, and $v_2$.

Since the $L'2$-vertices are also $L2$-vertices, the hypothesis holds for $G_2$ and $L'$ as long as every $L'2$-vertex in $G_2$ has at most one neighbor in $Q$.
Since $v_2$ in an $L2$-vertex it is not adjacent to any other $L2$-vertices.
If some $L'2$-vertex $x$ is adjacent to both $v_1$ and $v_0$, then the separating 2-chord $v_2v_1x$ contradicts our extremal choice of $Q$.

Hence by the minimality of $G$ there exists an $L'$-coloring $\psi$ of $G_2$ which agrees with $\varphi$ on $Q$ and together these colorings form an $L$-coloring of $G$.
\end{proof}

\begin{claim}\label{claim:2chordP}
\label{lastclaimc4}
$F$ does not contain a separating $2$-chord $v_0v_1v_2$ where $|L(v_2)| = 3$, $v_0\in V(P)$, and $v_0$ is not a middle vertex.
\end{claim}
\begin{proof}
Suppose there exists a separating 2-chord $Q=v_0v_1v_2 \subset G$ where $|L(v_2)| = 3$, $v_0 \in V(P)$, and $v_0$ is not a middle vertex.
Let $G_1$ and $G_2$ be the $Q$-components of $G$ where $G_1$ contains the vertices of $P$.

By the minimality of $G$, there exists an $L$-coloring $\varphi$ of $G_1$.
Let $L'$ be the list assignment on $G_2$ such that $L'(v_i) =\{ \varphi(v_i) \}$ for $i\in\{0,1,2\}$ and $L'(x) = L(x)$ for $x \in V(G_2) \setminus V(Q)$. 
The $L'1$-vertices in $G_2$ are exactly those in $Q$.

Since all $L'2$-vertices in $G_2$ are also $L2$-vertices, we must verify that every $L'2$-vertex in $G_2$ has at most one neighbor in $Q$.
If an $L'2$-vertex  $u$ has two neighbors, then one of them must be $v_1$ since $G$ is $C_4$-free. 
However, at least one of the 2-chords $v_0v_1u$ or $v_2v_1u$ is separating and contradicts Claim~\ref{claim:2chord}.

Hence there exists an $L'$-coloring $\psi$ of $G_2$ which agrees with $\varphi$ on $Q$ and together these colorings form an $L$-coloring of $G$.
\end{proof}

Our investigation of chords and 2-chords is complete.
We now investigate the lists of adjacent vertices along the outer face in Claims~\ref{cl:L2vertex} and \ref{cl:L3vertex}.

\begin{claim}\label{cl:L2vertex}
If $v_{0}v_1v_{2}$ is a path in $F$ where $|L(v_1)| = 2$,
then $L(v_1) \cap L(v_0) \neq L(v_1) \cap L(v_2)$.
\end{claim}
\begin{proof}
Suppose that there exists a path $v_0v_1v_2$ in $F$ where $L(v_1) = \{a,b\}$ and $L(v_1) \cap L(v_0) = L(v_1) \cap L(v_2) = \{a\}$.
We will find an $L$-coloring of $G$ where $v_1$ is assigned the color $b$.

Let $L'$ be the list assignment on $G-v_1$ where $L'(u) = L(u) \setminus\{b\}$ if $uv_1 \in E(G)$ and $L'(u) = L(u)$ otherwise.
Let $G'$ be obtained from $G-v_1$ by removing edges between $L'2$-vertices with disjoint lists.
We will verify that $G'$ and $L'$ satisfy the hypothesis.

If $u$ is a neighbor of $v_1$ with $b \in L(u)$, then $u$ is not in $F$ since by Claim~\ref{claim:chord} $G$ contains no chord $uv_1$.
Hence, the vertices that had the color $b$ removed are now $L'2$-vertices, all $L'2$-vertices are on the outer face of $G'$, and the $L'1$-vertices are exactly the vertices in $P$.

It remains to show that the $L'2$-vertices are independent in $G'$ and no $L'2$-vertex has two neighbors in $P$.
The $L2$-vertices in $G$ still form an independent set in $G'$.
The $L'2$-vertices that are neighbors of $v_1$ form an independent set since their $L'$-lists are pairwise disjoint (their $L$-lists previously contained $b$ and cannot share more colors). 
If an $L2$-vertex $u$ is adjacent to an $L'2$-vertex $x$ that is a neighbor of $v_1$, then since $u \notin \{v_0, v_2\}$, the path $uxv_1$ is a separating 2-chord contradicting Claim~\ref{claim:2chord}.
Similarly, if a neighbor $x$ of $v_1$ is adjacent to two vertices $u_0,u_1$ of $P$, then
at least one of them, say $u_1$, is not a middle vertex, and when $u_1 \notin \{v_0,v_2\}$ the path $v_1xu_1$ is a separating 2-chord contradicting Claim~\ref{claim:2chord}.
If $u_1 \in \{v_0,v_2\}$, then since $G$ contains no 4-cycles, the vertices $u_0$ and $u_1$ are adjacent and $v_1xu_0u_1$ is a 4-cycle.

Thus the hypothesis holds on $G'$ and $L'$, so by the minimality of $G$ there exists an $L'$-coloring $\varphi$ of $G'$ which extends to an $L$-coloring of $G$ with $\varphi(v_1) = b$.
\end{proof}

\begin{claim}\label{cl:L3vertex}
If $v_{0}v_1v_{2}$ is a path in $F$ where $|L(v_1)| = 3$, then $v_0$ and $v_2$ are $L2$-vertices, and the only $L2$-vertices adjacent to $v_1$.
\end{claim}
\begin{proof}
For a path $v_0v_1v_2$ where $|L(v_1)| = 3$, we consider how many of $v_0$ and $v_2$ are $L2$-vertices.

Suppose that neither $v_0$ nor $v_2$ is an $L2$-vertex. 
By Claim~\ref{claim:chord}, $G$ contains no good chord, and $G$ contains no bad chord $v_1u$  since $v_0$ and $v_2$ are not $L2$-vertices.
Thus, all neighbors of $v_1$ other than $v_0$ and $v_2$ are $L3$-vertices.
Select a color $a \in L(v_1)$ and let $L'$ be the list assignment on $G$ where $L'(z)=L(z)$ for $z \in V(G) \setminus \{v_1\}$ and $L'(v_1)=L(v_1)\setminus\{a\}$.
If $v_1$ is adjacent to two vertices of $P$, they are $v_0$ and $v_2$, and $F = P\cup\{v_1\}$.
This contradicts that $|V(F)| \geq 5$.
Thus, the hypothesis holds on $G$ with lists $L'$ and by the minimality of $G$ guarantees an $L'$-coloring of $G$, which is an $L$-coloring of $G$.

Now suppose that $v_2$ is an $L2$-vertex and $v_0$ is not.
By Claim~\ref{claim:chord}, $G$ contains no good chord, and if $G$ contains a bad chord $v_1u$ it is with a triangle $v_1uv_2$, and we can write $u = v_3$ as the other neighbor of $v_2$ on $F$; in this case, $v_3$ is an $L3$-vertex since it is adjacent to $v_2$.
Thus, all neighbors of $v_1$ other than $v_0$ and $v_2$ are $L3$-vertices.

Let $a$ be the color in $L(v_1) \cap L(v_2)$.
Let $G'$ be obtained from $G$ by removing the edge $v_1v_2$ and $L'$ be the list assignment where $L'(v_1) = L(v_1) \setminus \{a\}$ and $L'(x) = L(x)$ for $x \in V(G) \setminus \{v_1\}$.
Since the only $L2$-vertex adjacent in $G$ to $v_1$ is $v_2$, and they are not adjacent in $G'$, the $L'2$-vertices form an independent set in $G'$. 
Moreover, Claim~\ref{claim:chordP} implies that $v_1$ has at most one neighbor in $P$.
Hence $G'$ satisfies the hypothesis, and by the minimality of $G$ there exists an $L'$-coloring of $G'$. 
By the construction of $L'$ and $G'$,
$\varphi$ is also an $L$-coloring of $G$.

Thus, for a path $v_0v_1v_2$ in $F$ with $v_1$ an $L3$-vertex, $v_0$ and $v_1$ are both $L2$-vertices.
Since every bad chord $v_1u$ has $u$ adjacent to $v_0$ or $v_2$, the vertex $u$ is an $L3$-vertex.
Thus Claim~\ref{claim:chord} implies that $v_0$ and $v_2$ are the only $L2$-vertices adjacent to $v_1$.
\end{proof} 

By the minimality of the sum of the sizes of the lists, we can assume that $|V(P)| \geq 1$ by removing colors if necessary.
Let $p_0v_1v_2v_3\dots v_tv_{t+1}\dots$ be vertices of $F$ in cyclic order where $p_0 \in V(P)$, $\{ v_1,\dots,v_t\} = V(F)\setminus V(P)$, and thus $v_{t+1} \in V(P)$.

Claims~\ref{cl:L2vertex} and \ref{cl:L3vertex} together imply that for all $i \in \{1,\dots,t\}$, the vertex $v_i$ is an $L2$-vertex when $i$ is odd; otherwise $v_i$ is an $L3$-vertex.
Furthermore, $v_t$ is an $L2$-vertex, so $t$ is odd.

Select a set $X \subseteq \{v_2,v_3,v_4\}$ and a partial $L$-coloring $\varphi$ of $X$ by the following rules:
\begin{itemize}
\item[(X1)] If $v_2v_4$ is not a bad chord, then let $c \in L(v_2)\setminus(L(v_1)\cup L(v_3))$ and:
\begin{itemize}
\item[(X1a)]
If there is no common neighbor $w$ of $v_2$ and $v_3$ such that $c \in L(v_2)\cap L(w)$,
then let $X = \{v_2\}$ and $\varphi(v_2) = c$.
\item[(X1b)]
If there is a common neighbor $w$ of $v_2$ and $v_3$ 
such that $c \in L(v_2)\cap L(w)$, then let 
$X = \{v_2,v_3\}$, $\varphi(v_2) = c$,
and $\varphi(v_3) = b$ where $b$ is the unique color in $L(v_3) \setminus L(v_4)$.
\end{itemize}
\item[(X2)] If $v_2v_4$ is a bad chord, then let $X = \{v_2,v_3,v_4\}$. 
If $v_4$ and $v_5$ have a common neighbor $w$, then let $\varphi(v_4) \in L(v_4) \setminus (L(v_5)\cup L(w))$; otherwise let $\varphi(v_4) \in L(v_4) \setminus L(v_5)$.
Finally, select $\varphi(v_2) \in L(v_2)\setminus (L(v_1) \cup \{\varphi(v_4)\})$ and $\varphi(v_3) \in L(v_3)\setminus \{\varphi(v_4)\}$
such that $\varphi(v_2) \neq \varphi(v_3)$ 
\end{itemize}
Observe that $X$ and $\varphi$ are well-defined, since there is always a choice for $\varphi$ satisfying those rules.
See Figure~\ref{fig-defineX} for diagrams of these cases.

\begin{figure}[htp]
\begin{center}
 \includegraphics{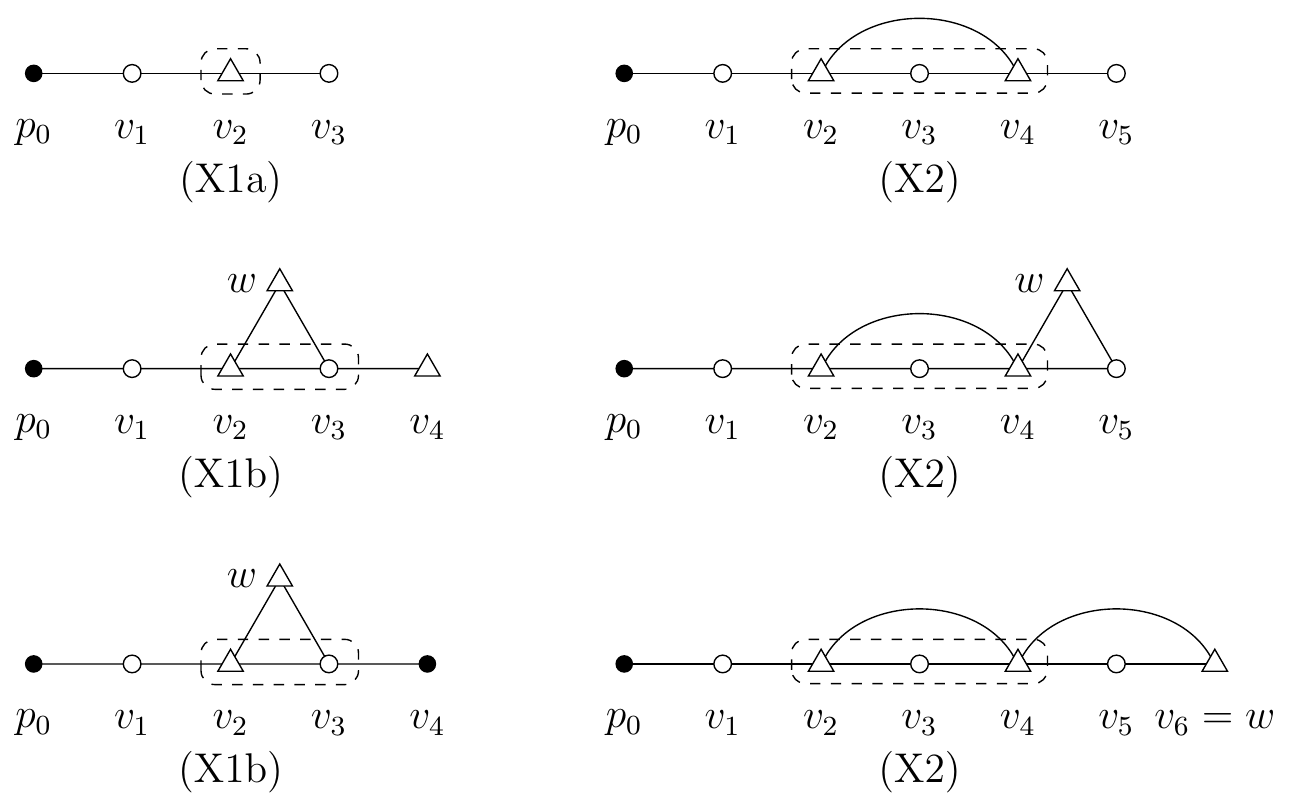}
\end{center}
\caption{Cases (X1) and (X2). A black circle is an $L1$-vertex, a white circle is an $L2$-vertex, and a triangle is an $L3$-vertex. The dashed box indicates $X$.}\label{fig-defineX}
\end{figure}

Let $L'$ be a list assignment on $G - X$ where 
\[ L'(v) = L(v) \setminus \{ \varphi(x) :  x \in X \text{ and } xv \in E(G)\}\]
for all $v \in V(G)\setminus X$.
Let $G'$ be obtained from $G-X$ by removing edges among vertices
with disjoint $L'$-lists except the edges of $P$.

Below, we verify that $G',L'$, and $P$ satisfy the assumptions of Theorem~\ref{thm:thomassen4cycle}. 
Then by the minimality of $G$, there is an $L'$-coloring $\psi$ of $G'$.
By the definition of $L'$, the colorings $\varphi$ and $\psi$ together form an $L$-coloring of $G$, a contradiction.

Let $N$ be the set of vertices $u$ where $|L(u)| > |L'(u)|$. 
Necessarily, every vertex of $N$ has a neighbor in $X$.
Observe that $X$ and $\varphi$ are chosen such that $L(u) = L'(u)$  for all $u \in V(F) \setminus X$.
Hence $N \subseteq V(G)\setminus V(F)$ and every vertex in $N$ is an $L3$-vertex.

Since $G$ is $C_4$-free, any pair of vertices has at most one common neighbor.
When $|X| = 3$, we are in the case (X2), and the chord $v_2v_4$ implies that no vertex in $N$ is adjacent to $v_3$, and a vertex adjacent to $v_2$ and $v_4$ would form a 4-cycle with $v_3$.
When $|X| = 2$, there is at most one vertex in $N$ having two neighbors in $X$.
This is possible only in the case (X1b), and the colors $\varphi(v_2)$ and $\varphi(v_3)$ are  chosen so that the common neighbor is an $L'2$- or $L'3$-vertex. 
Therefore $|L'(v)| \geq 2$ for every vertex $u \in N$.

If two vertices $x, y \in N$ are adjacent in $G'$, the color $c \in L(x) \cap L(y)$ is also in $L'(x) \cap L'(y)$ and hence the colors $a \in L(x) \setminus L'(x)$ and $b \in L(y) \setminus L'(y)$ are distinct.
Thus, $x$ is adjacent to some $v_i \in X$ where $\varphi(v_i) = a$, and $y$ is adjacent to some $v_j \in X$ where $\varphi(v_j) = b$.
In every case above, any two distinct vertices in $X$ that have neighbors not in $X$ are also adjacent, so $xv_iv_jy$ is a 4-cycle.
Thus, $N$ is an independent set.

Suppose that there is an edge $uv \in E(G')$ where $u \in N$ and $v \in V(F)\setminus X$ where $|L'(v)| = |L(v)| = 2$. 
If the 2-chord $xuv$ is separating, we find a contradiction by Claim~\ref{claim:2chord}.
If the 2-chord is not separating, then $x$ and $v$ are consecutive in $F$, and exactly one is in $X$.

First, we consider the case when $xuv = v_2uv_1$. 
If $L(v_1) \cap L(u) = L(p_0)$, then the edge $v_1u$ does not restrict the colors assigned to $v_1$ and $u$ by an $L$-coloring, so $G$ is not minimal; thus $L(v_1) \cap L(u) \neq L(p_0)$.
Hence the vertices $v_1$, $u$, and $v_2$ all share a common color, and this color was not removed from the list $L(u)$, so $|L'(u)| = 3$.

When $xuv \neq v_2uv_1$, then $xuv = v_iuv_{i+1}$, where $i$ is maximum such that $v_i \in X$.
However, the cases (X1a), (X1b), and (X2) all consider whether $v_i$ and $v_{i+1}$ have a common neighbor, and avoid using any color in common
if $v_{i+1}$ is an $L2$-vertex.
Therefore, $u$ is an $L'3$-vertex, so the $L'2$-vertices in $G'$ form an independent set.

Finally, we verify that no $L'2$-vertex in $G'$ has two neighbors in $P$.
Since $G'$ is obtained from $G$ by deletions of edges and vertices, it suffices to check the condition only for vertices in $N$.
If $v \in N$ has a neighbor $x \in X$, then $v$ is not adjacent to two vertices of $P$ by Claims~\ref{claim:2chord} and~\ref{claim:2chordP}
and $G$ being $C_4$-free.

Therefore, $G',L'$, and $P$ satisfy the assumptions of Theorem~\ref{thm:thomassen4cycle}.
\end{proof}


\section{Forbidding 5- and 6-cycles}

The goal of this section is to prove Theorem~\ref{c56free}.
We prove a slightly stronger statement.

\begin{theorem}~\label{thm:c5c6_up}
Let $G$ be a plane graph without $5$- or $6$-cycles and let $p \in V(G)$.
Let  $L$ be a $(*,1)$-list assignment such that 
\begin{itemize}
 \item $|L(p)| = 1$,
 \item $|L(v)| = 3$ for $v \in V(G)-p$.
\end{itemize}
Then $G$ is $L$-colorable.
\end{theorem}

This strengthening allows us to assume that a minimum counterexample is 2-connected, since we can iteratively color a graph by its blocks using at most one precolored vertex at each step.

Our proof uses a discharging technique.
In Section~\ref{subsec:config}, we define a family of \emph{prime graphs} and prove in Section~\ref{subsec:reduce} that a minimum counterexample is prime.
The proof is then completed in Section~\ref{subsec:discharge}, where we define a discharging process and prove that prime graphs do not exist, and hence a minimum counterexample does not exist.

\subsection{Configurations}\label{subsec:config}

We introduce some notation for a plane graph $G$.
Let $V(G)$, $E(G)$, and $F(G)$ be the set of vertices, edges, and faces, respectively.
For $v\in V(G)$, let $d(v)=|N(v)|$ where $N(v)$ is the set of vertices adjacent to $v$. 
For $f\in F(G)$, let $d(f)$ be the length of $f$. 

For a $C_5$- and $C_6$-free plane graph, the subgraph of the dual graph induced by the 3-faces has no component with more than three vertices.
A \emph{facial $K_4$} is a set of three pairwise adjacent 3-faces.
We say four vertices $xz_1yz_2$ form a \emph{diamond} if $xz_1yz_2$ is a 4-cycle formed by two adjacent 3-faces $xyz_i$ for $i \in \{1,2\}$.
If a 3-face is not adjacent to another 3-face, then it is \emph{isolated}.

A vertex is \emph{low} if it has degree three; otherwise it is \emph{high}.
A 3-face is \emph{bad} if it is incident to a low vertex; otherwise it is \emph{good}.
A face is \emph{small} if it has length three or four.
A face is \emph{large} if it has length at least seven.
A 4-face is \emph{special} if is is incident to $p$ and \emph{normal}
otherwise.

\begin{definition}\label{def:prime}
For a plane graph $G$, a list assignment $L$ from Theorem~\ref{thm:c5c6_up}, and $p\in V(G)$ with $|L(p)|=1$, the pair $(G,L)$ is \emph{prime} if 
\begin{itemize}
\item $G$ is 2-connected

\item $d(v) \geq 3$ for every $v \in V(G)- p$

\item $d(p) \geq 2$
\end{itemize}

and in addition $(G,L)$ contains none of the configurations (C\ref{conf:pC3})--(C\ref{conf:last}) below:

\begin{enumerate}[(C1)]

\item\label{conf:pC3} A $3$-face containing $p$.

\item\label{conf:lowface} A normal $4$-face where all incident vertices are low.

\item\label{conf:c3-1high} A 3-face incident to at most one high vertex.

\item\label{conf:baddiamond} A diamond  $xz_1yz_2$ where $d(z_1) = d(z_2) =3$ and $d(x)=d(y) = 4$.

\item\label{conf:crown} A diamond  $xz_1yz_2$ where $d(y) = 4$ and $d(x) = 3$.

\item\label{conf:facialk4} 
A facial $K_4$ $wxyz$ where $w$ is the internal vertex, and at least one of $x$, $y$, and $z$ has degree at most $4$.
\label{conf:lastsimple}

\begin{figure}[tp]
\begin{center}
\includegraphics{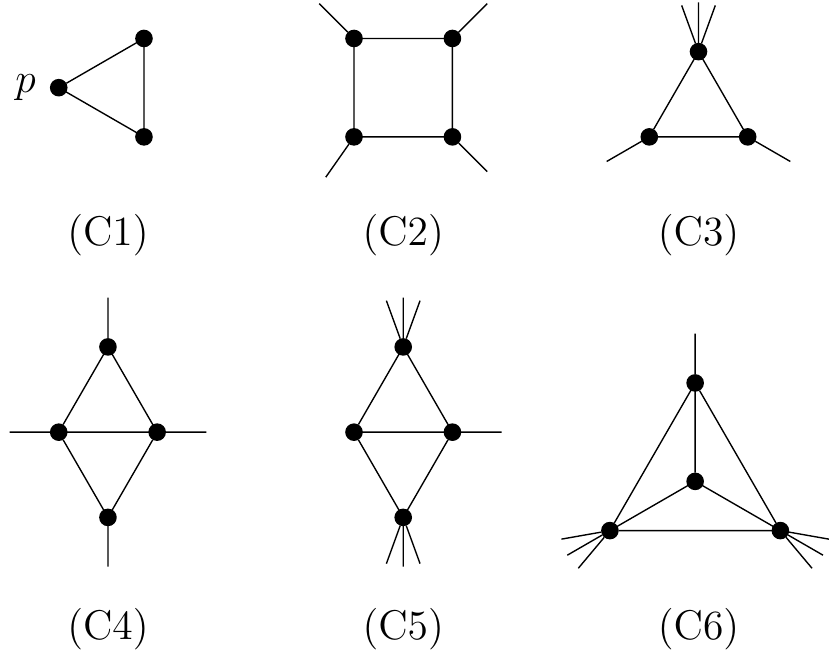}
\caption{Simple reducible configurations.}
  \label{fig-reducible}
\end{center}
\end{figure}

\begin{figure}[tp]
\begin{center}
\includegraphics{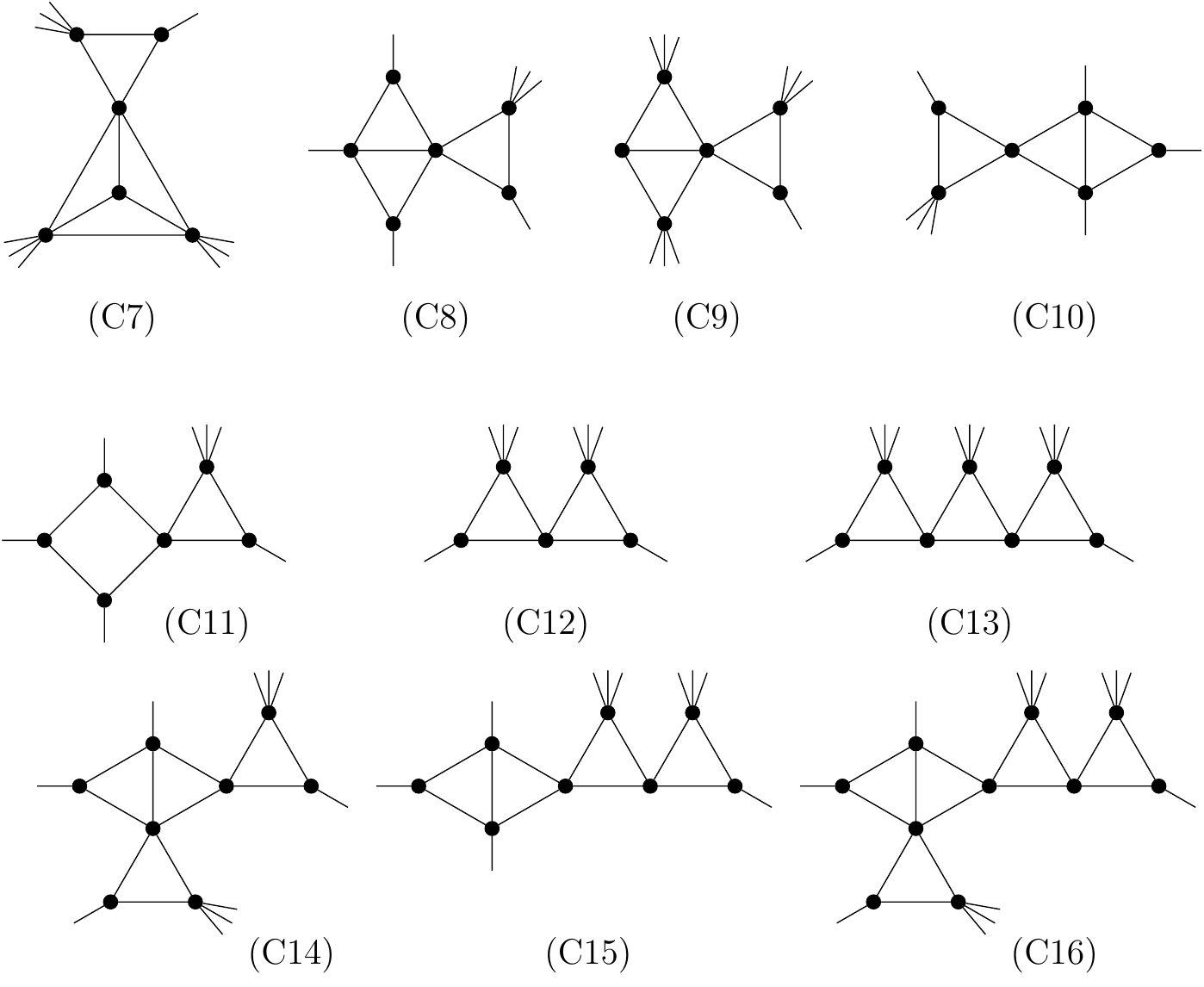}
\caption{Compound reducible configurations.}
  \label{fig-reducible56plus}
\end{center}
\end{figure}

\item\label{conf:facialk4c3} A facial $K_4$ $wxyz$ where $w$ is the internal vertex, the vertex $z$ has degree exactly five, and the other two neighbors $u,v$ of $z$ bound a bad 3-face $zvu$.
\label{conf:firstcompound}

\item\label{conf:diamondc3a} A diamond  $xz_1yz_2$ where $d(z_1)=d(z_2) = 3$, $d(x) = 4$, $d(y) = 5$, and the other two neighbors $u,v$ of $y$ form a bad 3-face $yvu$.

\item\label{conf:diamondc3b} A diamond  $xz_1yz_2$ where $d(x) = 3$, $d(y) = 5$, and the other two neighbors $u,v$ of $y$ form a bad 3-face $yvu$.

\item\label{conf:diamondc3c} A diamond  $xz_1yz_2$ where $d(z_2) = 3$, $d(x) = d(y) = d(z_1) = 4$, and the other two neighbors $u,v$ of $z_1$ form a bad 3-face $z_1vu$.

\item\label{conf:squareplusc3} A bad 3-face $xyz$ and a normal 4-face $wuvx$ where $d(x)=4$ and $x$ is the only high vertex incident to the 4-face.

\item\label{conf:c3c3} Two 3-faces $xyz$ and $xuv$ where $d(x)=4$ and $d(y) = d(v) = 3$.

\item\label{conf:threec3s} 
Three 3-faces $xyz$, $xuv$, $vpq$, where $d(y) = d(p) = 3$ and $d(x) = d(v) = 4$.

\item\label{conf:diamondNEW2c3s} A diamond $xz_1yz_2$ where $xy$ is an edge, $d(z_1) = 3$, $d(y) = 4$, $d(x) = 5$, and $d(z_2) = 4$, where $x$ and $z_2$ are each incident to a bad 3-face.

\item\label{conf:diamondNEWc3chain} A diamond $xz_1yz_2$ where $xy$ is an edge, $d(z_1) = 3$ and $d(y) = d(x) = d(z_2) = 4$, where $z_2$ is incident to a good 3-face $z_2uv$ with $d(v) = 4$ and $v$ is incident to another bad 3-face.

\item\label{conf:diamondNEWc3chainPLUS} A diamond $xz_1yz_2$ where $xy$ is an edge, $d(z_1) = 3$,  $d(x) = 5$, and $d(y) = d(z_2) = 4$, where $x$ is incident to a bad 3-face and $z_2$ is incident to a good 3-face $z_2uv$ with $d(v) = 4$ and $v$ is incident to another bad 3-face.

\label{conf:last}
\end{enumerate}
\end{definition}

The configurations (C1)--(C\ref{conf:lastsimple}) are called \emph{simple}.
See Figure~\ref{fig-reducible}. 
Other configurations can be built from simple ones by replacing an edge with one endpoint in the configuration by a bad 3-face; we call these \emph{compound}. 
See Figure~\ref{fig-rules-generating-noted} for a sketch of creating compound configurations. 
For convenience, we list compound configurations used in our proof. See Figure~\ref{fig-reducible56plus}.  
Reducibility is proved in Lemma~\ref{glue} from Section~\ref{subsec:reduce}.

Observe that a prime graph $G$ has no 5- or 6-faces since no 5- or 6-cycles exist and $G$ is 2-connected.

\subsection{Discharging}\label{subsec:discharge}

In this section, we prove the following proposition.

\begin{proposition}\label{prop:noprime}
	No pair $(G,L)$ is prime.
\end{proposition}

We shall prove that a prime $(G,L)$ does not exist by assigning an initial \emph{charge} $\mu(z)$ to each $z\in V(G) \cup F(G)$ with strictly negative total sum, then applying a discharging process to end up with charge $\mu^*(z)$.
We prove that since $(G,L)$ does not contain any configuration in (C\ref{conf:pC3})--(C\ref{conf:last}), then $\mu^*$ has nonnegative total sum.
The discharging process will preserve the total charge sum, and hence we find a contradiction and $G$ does not exist.

For every vertex $v\in V(G)- p$ let $\mu(v)=2d(v)-6$, for $p$ let $\mu(p)=2d(p)$, and for every face $f\in F(G)$, let $\mu(f)=d(f)-6$. 
The total initial charge is negative by
\begin{align*}
\sum_{z\in V(G)\cup F(G)} \mu(z)
	&=\sum_{v\in V(G)- p} (2d(v)-6)+2d(p)+\sum_{f\in V(F)} (d(f)-6) \\
	&=6|E(G)|-6|V(G)|-6|F(G)|+6
	=-6.
\end{align*}
The final equality holds by Euler's formula.

In the rest of this section we will prove that the sum of the final charge after the discharging phase is nonnegative. 
Instead of looking at each individual face, we look at groups of adjacent 3-faces. 

Note that since $G$ has no $5$-cycles and $6$-cycles,  no 4-face is adjacent to a 3- or 4-face
(and hence every face adjacent to a 4-face has length at least seven).
If a vertex $v$ with $d(v) \geq 4$ is incident to $\ell_3$ 3-faces and $\ell_4$ 4-faces, then $d(v) \geq \frac{3}{2}\ell_3 + 2\ell_4$.
Thus, every vertex $v$ is incident to at most $2d(v)/3$ small faces. 

\begin{figure}[tp]
\begin{center}
\includegraphics[scale=0.90]{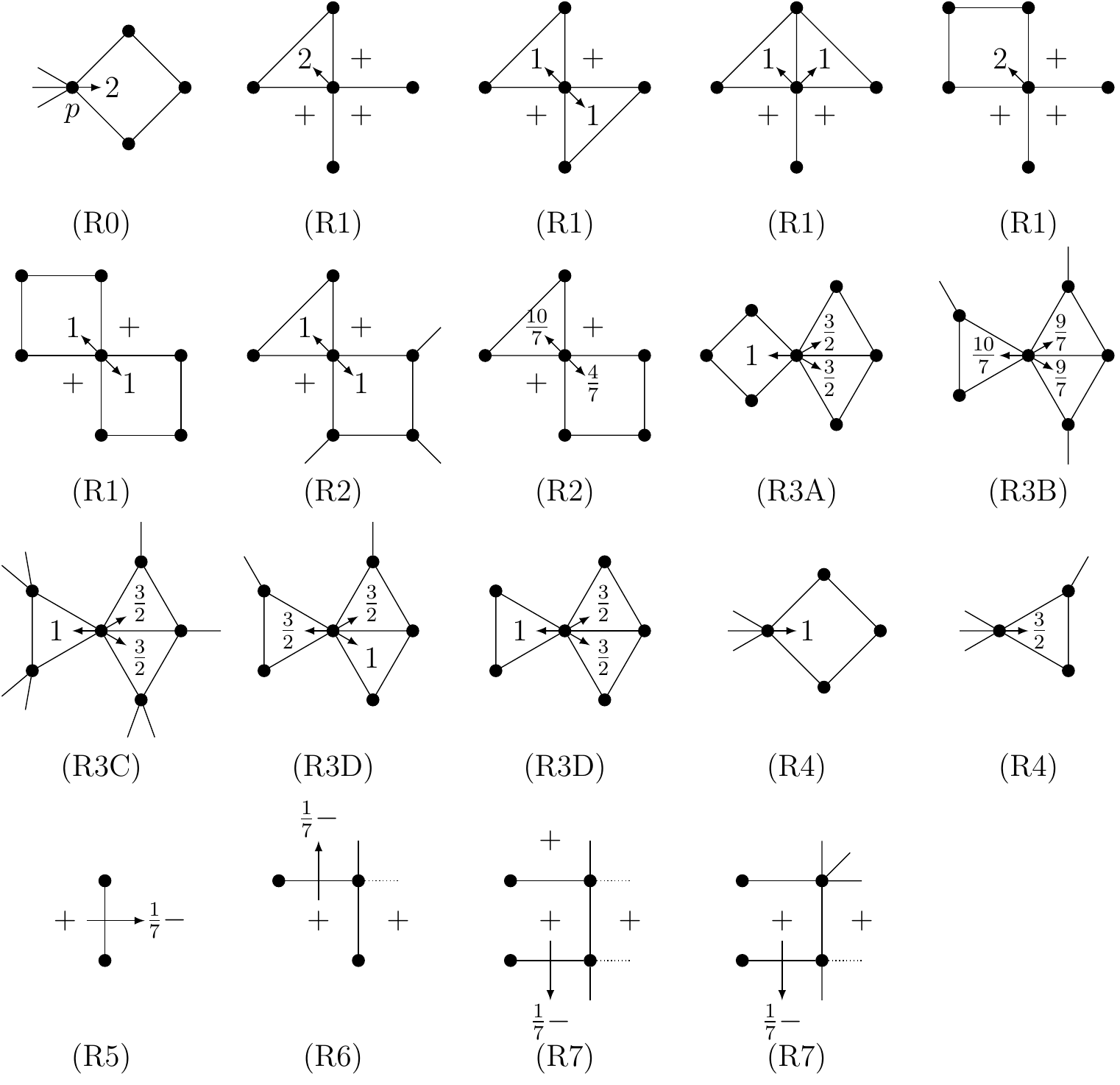}
\caption{Discharging rules.}
  \label{fig:discharging}
\end{center}
\end{figure}

We begin by discharging from vertices with positive charge to small faces with negative charge.
The precolored vertex $p$ transfers charge according to rule (R0).
\begin{enumerate}
\item[(R0)] $p$ sends charge 2 to every (special) incident $4$-face.
\end{enumerate}

For a vertex $v \in V(G) - p$ with $d(v) \geq 4$, exactly one of the discharging rules (R1)--(R4) applies; rules (R0)--(R4) are called \emph{vertex rules}.

\begin{enumerate}[(R1)]
\item[(R1)] If $d(v)=4$ and $v$ is not incident to both a 3-face and a normal $4$-face, then $v$ distributes its charge uniformly to each incident 3-face or normal $4$-face. 

\item[(R2)] If $d(v)=4$ and $v$ is incident to a 3-face $t$ and a normal $4$-face $f$, then:

\begin{enumerate}[(R2A)]
\item If $f$ is incident to exactly one high vertex, then $v$ gives charge $1$ to $f$ and $1$ to $t$.

\item If $f$ is incident to more than one high vertex, then $v$ gives charge ${4\over 7}$ to $f$ and ${10\over 7}$ to $t$.
\end{enumerate}

\item[(R3)] If $d(v) = 5$, then:
\begin{enumerate}[(R3A)]
\item If $v$ is incident to a normal $4$-face, then $v$ gives charge $1$ to each normal $4$-face and distributes its remaining charge uniformly to each incident 3-face.

\item If $v$ is incident to three bad 3-faces, then $v$ gives charge ${10\over 7}$ to the isolated 3-face and ${9\over 7}$ to each 3-face in the diamond.
 
\item If $v$ is incident to only one bad 3-face that is in a diamond with another 3-face incident to $v$, then $v$ gives charge ${3\over 2}$ to both 
3-faces in the diamond, and if $v$ is incident to another 3-face $t$, then $v$ gives charge $1$ to $t$.

\item Otherwise, $v$ gives charge ${3\over 2}$ to each incident bad 3-face and distributes its remaining charge uniformly to each incident non-bad 3-face.
\end{enumerate}

\item[(R4)] If $d(v)\geq 6$, then $v$ gives charge $1$ to each incident normal $4$-face and charge ${3\over 2}$ to each incident 3-face.
 
\end{enumerate}

After applying the vertex rules, we say a face is \emph{hungry} if it is a negatively-charged small face, or it is a 3-face in a negatively-charged diamond.

We now discharge from large faces to hungry faces; the rules (R5)--(R7) are \emph{face rules}.
Let $f$ be a face with $d(f)\geq 7$ and let $f_0, f_1, f_2, \ldots,f_{d(f)} = f_0$ be the faces adjacent to $f$ in counterclockwise order.
Observe that $f$ has charge at least $d(f)/7$, and so $f$ could send charge $\frac{1}{7}$ to each adjacent face.
Each of the rules below could apply to $f$ and an adjacent face $f_i$, to decide where the charge $\frac{1}{7}$ associated with $f_i$ should go.

\begin{enumerate}[(R5)]
\item[(R5)] If $f_i$ is hungry, then $f$ gives charge ${1\over 7}$ to $f_i$.

\item[(R6)] If $f_i$ is not hungry, $f_{i+1}$ is hungry, and the vertex incident to $f$, $f_i$, and $f_{i+1}$ has degree at most four, then $f$ gives charge ${1\over 7}$ to $f_{i+1}$ instead of $f_i$.

\item[(R7)] If $f_{i}$ is not hungry, $f_{i-1}$ is hungry, the vertex incident to $f$, $f_{i-1}$, and $f_{i}$ has degree at most four, and either the vertex incident to $f$, $f_i$, and $f_{i+1}$ has degree at least five or $f_{i+1}$ is not hungry, then $f$ gives charge ${1\over 7}$ to $f_{i-1}$ instead of $f_i$.
\end{enumerate}

We now show that the discharging rules result in a nonnegative charge sum $\sum_{v\in V(G)} \mu^*(v) + \sum_{f\in F(G)} \mu^*(f) \geq 0$, contradicting our previously computed sum of $-6$.
First, we prove that the final charge $\mu^*$ is nonnegative on every vertex.
Then, we prove that the final charge $\mu^*$ is nonnegative on every large face and every 4-face.
A set $S$ of 3-faces is \emph{connected} if they induce a connected subgraph of the dual graph.
Since $G$ contains no 5- or 6-cycles, a connected set of 3-faces is either a facial $K_4$, a diamond, or an isolated 3-face.
We will show that for every connected set $S$ of 3-faces, the final charge sum $\sum_{f \in S} \mu^*(f)$ is nonnegative.

\begin{claim}
For each vertex $v\in V(G)$, the final charge $\mu^*(v)$ is nonnegative.
\end{claim}

\begin{proof}
If $v=p$, then rule (R0) applies then $p$ is incident to at most $d(p)/2$
4-faces since 4-faces cannot share an edge. 
So $\mu^*(v) \geq 2d(p) - 2d(p)/2 \geq 0$.

Assume $v \in V(G)\setminus p$.
Recall $d(v) \geq 3$.
If $d(v) = 3$, then $\mu(v) = 0$ and no charge is sent from this vertex.

If $d(v)=4$, then $\mu(v) = 2$ and (R1) or (R2) applies.
Consider the four faces incident to $v$.
Since $G$ avoids $5$- and 6-cycles, at most two of these faces are small.
If $v$ is not incident to both a 3- and $4$-face, then (R1) applies and $v$ sends all charge uniformly to each small face; hence $\mu^*(v) = 0$.
If $v$ is incident to both a 3- and $4$-face, then (R2) applies and $v$ sends total charge two to the two small faces (either as $1+1$ for (R2A) or $\frac{4}{7} + \frac{10}{7}$ for (R2B)).

If $d(v)=5$, then $\mu(v)=4$ and (R3) applies.
There are five faces incident to $v$, and since $G$ avoids 5- and 6-cycles, at most three of these faces are small.
If $v$ is incident to three bad 3-faces, then two of the faces are adjacent so these faces partition into a bad 3-face and a diamond; (R3B) applies and a total charge of four is sent from $v$, so $\mu^*(v) = 0$. 
If $v$ is not incident to three bad 3-faces, $v$ is incident to at most two 3- or 4-faces;  (R3A), (R3C), or (R3D) applies, and $v$ sends at most charge three, $\mu^*(v) \geq 0$.

If $d(v)\geq 6$, then (R4) applies.
Since $G$ avoids 5- and 6-cycles, $v$ is incident to at most $\frac{2d(v)}{3}$ small faces.
Since $\mu(v) = 2d(v)-6$ and $v$ sends charge at most $\frac{2d(v)}{3}\cdot \frac{3}{2} = d(v)$, the final charge on $v$ is $\mu^*(v) \geq d(v)-6 \geq 0$.
\end{proof}

\begin{claim}
For each face $f\in F(G)$ with $d(f)\geq 7$, the final charge $\mu^*(f)$ is nonnegative.
\end{claim}

\begin{proof}
Let the faces adjacent to $f$ be listed in clockwise order as $f_1,f_2,\dots$ as in the discharging rules.
Observe that each adjacent face $f_i$ satisfies at most one of the rules (R5), (R6), and (R7), and hence $f$ sends charge $\frac{1}{7}$ at most $d(f)$ times, leaving $\mu^*(f) \geq \mu(f) - \frac{d(f)}{7} \geq 0$.
\end{proof}

\begin{claim}
For each $4$-face $f\in F(G)$, the final charge $\mu^*(f)$ is nonnegative.
\end{claim}
\begin{proof}
If $f$ is a special $4$-face, (R0) applies. Thus
$\mu^*(f) =-2 + 2=0$.
So assume that $f$ is a normal $4$-face.

Observe $\mu(f)=-2$, and all faces adjacent to $f$ have length at least seven.
Since $G$ contains no (C\ref{conf:lowface}), the normal 4-face $f$ is incident to at least one high vertex.

If $f$ is incident to exactly one high vertex $v$, then $v$ sends charge at least $1$ to $f$ by (R1), (R2A), or (R4).
The four incident faces each send charge at least $\frac{1}{7}$ by (R5). 
Three of the four faces adjacent to $f$ each send an additional $\frac{1}{7}$ by (R6).
Thus, $\mu^*(f)\geq -2+1+4\cdot\frac{1}{7}+3\cdot\frac{1}{7}=0$.

If $f$ is incident to exactly two high vertices $u$ and $v$, then $u$ and $v$ each send charge at least $\frac{4}{7}$ by (R1), (R2B), (R3A), or (R4). 
The four incident faces each send charge $\frac{1}{7}$ by (R5).
Two  of the four faces adjacent to $f$ each send an additional $\frac{1}{7}$ by (R6).
Thus, $\mu^*(f)\geq -2+2\cdot{4\over 7}+4\cdot{1 \over 7}+2\cdot{1\over 7}=0$.

If $f$ is incident to at least three high vertices, then each high vertex sends charge at least $4\over 7$ by (R1), (R2B), (R3A), or (R4). 
The four faces adjacent to $f$ each send charge at least $1\over 7$ by (R5).
Thus, $\mu^*(f)\geq -2+3\cdot{4\over 7}+4\cdot{1\over 7}\geq0$.
\end{proof}

We now show the total charge sum over the 3-faces is nonnegative by showing the charge sum is nonnegative on each connected set of 3-faces, starting with facial $K_4$'s (Claim~\ref{k4}), then diamonds (Claim~\ref{dia2}), and finally isolated 3-faces (Claim~\ref{claim:isotri}).

Observe that if $t$ is a good 3-face, then $t$ receives charge at least 1 from every incident vertex by the vertex rules.

\begin{claim}\label{k4}
For each facial $K_4$, the sum of the final charge of the three 3-faces is nonnegative.
\end{claim}
\begin{proof}
Let $wxyz$ be a facial $K_4$ where $w$ is the internal vertex.
Since $G$ contains no (C\ref{conf:facialk4}), all vertices $v \in \{x,y,z\}$ have degree $d(v) \geq 5$.
Since $G$ contains no (C\ref{conf:facialk4c3}), any vertex $v \in \{x,y,z\}$ with $d(v) = 5$ is not incident to another bad 3-face outside the facial $K_4$.
Thus, each vertex $v\in\{x, y, z\}$ sends charge at least $2\cdot\frac{3}{2}$ by (R3D) or (R4) to the 3-faces in the facial $K_4$, and $\mu^*(wxy)+\mu^*(wyz)+\mu^*(wzx) \geq -9 + 3 + 3 + 3 = 0$.
\end{proof}

\begin{claim}\label{dia2}
For each diamond, the sum of the final charge of the two 3-faces is nonnegative.
\end{claim}

\begin{proof}
Let the diamond have vertices $xz_1yz_2$ where $xyz_i$ is a 3-face $f_i$ for $i \in \{1,2\}$.
We assume $d(x) \leq d(y)$.
Note that if the diamond does not have nonnegative charge after the vertex rules, then both faces $f_1$ and $f_2$ are hungry.

By our earlier observation, if both faces $f_1$ and $f_2$ are good, then they each receive charge at least 3 from the incident vertices, and the diamond has nonnegative charge after the vertex rules.
We now consider which faces in $f_1$ and $f_2$ are bad.

\begin{mycases}

\mycase{Exactly one 3-face $f_i$ is bad.} 
In this case, we will assume $d(z_1) \leq d(z_2)$, so it must be $f_1$ that is bad, while $f_2$ is good.
Thus, $d(y) \geq d(x) \geq 4$, and $d(z_2) \geq 4$.
Since $f_1$ is bad, $d(z_1) = 3$.
For $v \in \{x,y\}$, let $f_v$ be the face incident to $v$ that follows $f_1$ and $f_2$ in the counterclockwise order around $v$.

If $d(y)\geq 5$ then $y$ contributes at least $\frac{3}{2}+1$ by (R3D), $2\cdot \frac{9}{7}$ by (R3B), or at least $2\cdot \frac{3}{2}$ by (R3A), (R3C), or (R4).
If $d(y) = 4$ then $y$ contributes at least $2$ by (R1), but now (R6) also applies to the face $f_y$, adding an extra contribution of $\frac{1}{7}$. 
The contribution to the diamond is at least $2+\frac{1}{7}$ from the vertex rules applied to $y$ and (R6) applied to the face $f_y$.
By symmetry, the contribution to the diamond is also at least $2+{1\over 7}$ from the vertex rules applied to $x$ and (R6) applied to the face $f_x$. 
By the vertex rules, $z_2$ sends charge at least 1 to $f_2$.
Rule (R5), the four faces adjacent to the diamond send $\frac{1}{7}$ each.
Rule (R6) applies to the face incident to $z_1$ that is not $f_1$, $f_x$, or $f_y$, giving $\frac{1}{7}$ to the diamond.
Thus $\mu^*(f_1)+\mu^*(f_2) \geq -6 + 2(2 + \frac{1}{7}) + 1 + 4\cdot{1 \over 7}  + \frac{1}{7} \geq 0$.

\mycase{Both 3-faces $f_1$ and $f_2$ are bad.}
We consider the degree of $x$ and order $z_1$ and $z_2$ such that $z_1$, $y$, and $z_2$ appear consecutively in the clockwise ordering of the neighbors of $x$.
Observe that when $d(x) > 3$, we have $d(z_1)=d(z_2)=3$ since $f_1$ and $f_2$ are bad.
\begin{subcases}
\subcase{$d(x) \geq 5$.}
By (R3) and (R4), both $x$ and $y$ each send at least $2\cdot\frac{9}{7}$. 
By (R5), the four incident faces contribute charge $4\cdot\frac{1}{7}$ to the diamond and by (R6), two of the faces incident to $z_1$ or $z_2$ contribute at least $2\cdot\frac{1}{7}$.
Thus the final charge on the diamond is $\mu^*(f_1)+\mu^*(f_2) \geq -6  + 4\cdot\frac{9}{7} +4\cdot\frac{1}{7} +2\cdot\frac{1}{7} = 0$.

\subcase{$d(x) = 4$.}
By (R1), the vertex $x$ sends charge $1$ to each face $f_i$.
Since $G$ contains no (C\ref{conf:diamondc3a}), if $d(y)=5$ then $y$ is not incident to a bad 3-face other than $f_1$ and $f_2$.
Thus, $y$ sends charge $\frac{3}{2}$ to each face $f_i$, and after the vertex rules the charge on the diamond is $-6 + 2\cdot 1 + 2\cdot\frac{3}{2} = -1$, and the faces $f_i$ are hungry.
By (R5), the four faces adjacent to $f_1$ and $f_2$ each send charge $1\over 7$ to the diamond.
By (R6), three of the four faces adjacent to $f_1$ and $f_2$  each send charge $\frac{1}{7}$ to the diamond.
Thus, $\mu^*(T_1)+\mu^*(T_2)=-6+2\cdot 1+2\cdot\frac{3}{2}+4\cdot{1\over 7}+3\cdot{1\over 7}\geq 0$.

\subcase{$d(x) = 3$.}
Since $G$ contains no (C\ref{conf:crown}), we have $d(y) \geq 5$.
Let $f$ be the face incident to $z_1, x, z_2$.
If $f$ is a 3-face, then $z_1xz_2y$ is a facial $K_4$, handled in Claim~\ref{k4}.
If $f$ is a 4-face, then $G$ contains a 5-cycle, a contradiction.
Thus, $d(f) \geq 7$, and let $s_0, s_1, s_2, \ldots$ be the faces adjacent to $f$ in cyclic clockwise order where $s_1=f_1$ and $s_2=f_2$.
Since $G$ contains no (C\ref{conf:diamondc3b}), $d(y) = 5$ implies that the vertex $y$ is not adjacent to a bad 3-face other than $f_1$ and $f_2$.
By (R3C), (R3D), and (R4), $y$ sends charge $\frac{3}{2}$ to each face $f_i$. 
By (R5), four faces adjacent to the diamond each contribute at least ${1\over 7}$. 
Since $G$ contains no (C\ref{conf:c3-1high}), it follows that $d(z_i)\geq 4$ for each $i\in\{1, 2\}$.

If $d(z_i)\geq 5$ for some $i\in \{1, 2\}$, then by the vertex rules, $z_1,z_2$ together will contribute at least $1+{10\over 7}$ to the diamond. 
Thus, $\mu^*(f_1)+\mu^*(f_2)\geq -6 + 1+ {10\over 7} + 3 + 4\cdot {1\over 7} \geq 0$. 
We can now assume $d(z_1) = d(z_2) = 4$.

\subsubcase
If $d(s_3) \geq 7$, then $z_2$ sends charge $2$ to $f_2$ by (R1)
and $z_1$ sends charge at least $1$ to $f_1$ by the vertex rules.
Thus $\mu^*(f_1)+\mu^*(f_2)\geq -6 + 2\cdot\frac{3}{2} + 2 + 1 \geq 0$.
By symmetry this also solves the case when $d(s_0)\geq 7$.

\subsubcase If $d(s_3) = 4$, then since $s_3$ and $s_2$ are not a copy of (C\ref{conf:squareplusc3}), the face $s_3$ must be incident to at least two high vertices.
By (R2B), $z_3$ sends charge $10\over 7$ to $f_2$ and by the vertex rules, $z_1$ sends charge at least $1$ to $f_2$.
Thus, $\mu^*(f_1)+\mu^*(f_2)\geq -6+2\cdot\frac{3}{2}+{10\over 7}+1+4\cdot{1\over 7}=0$. 
By symmetry this also solves the case when $d(s_0) = 4$.

\subsubcase If $d(s_0) = d(s_3) = 3$, then since the faces $s_0$ and $f_1$ (or the faces $f_2$ and $s_3$) do not form a copy of (C\ref{conf:c3c3}), the 3-faces $s_0$ and $s_3$ are not bad 3-faces.

Let $z_0$ be the vertex such that $z_0z_1$ is the edge between $f$ and $s_0$; similarly let $z_3$ be the vertex such that $z_3z_2$ is the edge between $f$ and $s_3$. 
Let $w_i$ be the other vertex of $s_i$ different from $z_i$ for $i \in \{0,3\}$.
For $i \in \{0,3\}$, we have $d(z_i) \geq 4$ and $d(w_i) \geq 4$ since $s_i$ is not a bad 3-face.
See Figure~\ref{fig-discharginghelp} for a sketch of the situation.

\begin{figure}[tp]
\begin{center}
\includegraphics{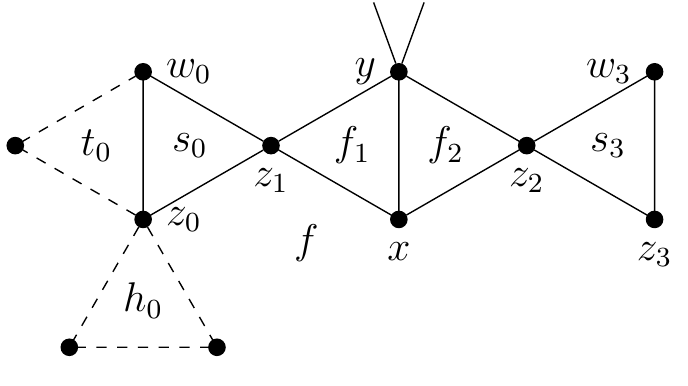}
\caption{Situation in Case 3.iii(c).}
  \label{fig-discharginghelp}
\end{center}
\end{figure}

Let $i$ be in $\{0,3\}$.
If $s_i$ is an isolated 3-face, then it receives charge at least $1$ from each of its incident vertices and is not hungry after the vertex rules.
If $s_i$ is in a diamond with a bad 3-face $t_i$ then since $t_i,s_i,f_{|i-1|}$ is not a copy of (C\ref{conf:diamondc3c}), at least one of $z_i$ and $w_i$ has degree at least 5. By symmetry, assume $z_i$ has degree at least $5$. 
If (R3D) does not apply to $z_i$, then the diamond formed by $s_i,t_i$ receives charge at least 6 from its vertices and is not hungry after the vertex rules.
Hence (R3D) applies to $z_i$ and there must be a bad 3-face $h_i$ incident to $z_i$ that is not $t_i$.
If $w_i$ has degree at least 5, by symmetry, (R3D) applies and $t_i,s_i$ are not hungry. If $w_i$ has degree 4 then the faces $s_i$, $t_i$, $h_i$, and $f_{|i-1|}$ form a copy of (C\ref{conf:diamondNEW2c3s}).
Therefore the diamond $s_i,t_i$ is not hungry after the vertex rules. 
Therefore $s_0$ and $s_3$ are not hungry after the vertex rules.

By (R5), the three faces adjacent to $f_1$ and $f_2$ send charge $4\cdot\frac{1}{7}$ to the diamond.
The rule (R6) applied on $f$ and edge $z_2z_3$ sends charge $\frac{1}{7}$ to $f_2$
and (R6) applied on edge $z_1w_0$ and face containing $z_1w_0$ sendscharge $\frac{1}{7}$ to $f_1$.
Finally, we show that (R7) applies on $z_1z_0$ which gives additional $\frac{1}{7}$. 
If $d(z_0) \geq 5$ then (R7) applies. Suppose that $d(z_0) = 4$. If $s_0$ is in a diamond with $t_0$, then (R6) cannot apply on $z_1z_0$.
If $z_0$ is in another bad 3-face $h_0$, then $h_0,s_0,f_1$ form reducible configuration (C\ref{conf:threec3s}). Hence (R7) indeed applies.
Thus,  $\mu^*(f_1)+\mu^*(f_2)\geq -6 + 2\cdot\frac{3}{2} + 1 + 1 + 4\cdot{1\over 7}+2\cdot{1\over 7}+\frac{1}{7} \geq 0$. 
\end{subcases}
\end{mycases}
In all cases, our diamond has nonnegative total charge.
\end{proof}

\begin{claim}\label{claim:isotri}
For each isolated 3-face $t$, the final charge $\mu^*(t)$ is nonnegative.
\end{claim}
\begin{proof}
Note that $\mu(t)=-3$.
If $t$ is good, then each incident vertex sends charge at least $1$ by the vertex rules and $\mu^*(t)\geq -3+3=0$.
We now assume $t$ is incident to at least one low vertex. 
Moreover, we assume that $t$ is hungry.

Since $G$ contains no (C\ref{conf:c3-1high}), $t$ is incident to exactly one low vertex.
Let $x$, $y$, and $z$ be the vertices incident to $t$ in counter-clockwise order where $x$ is low, and let $f$ be the face sharing the edge $yz$ with $t$.
The three faces adjacent to $t$ each send charge $\frac{1}{7}$ to $t$ by (R5).
The faces having $zx$ in common with $t$ sends charge $\frac{1}{7}$ to $t$ by (R6).

If one of $y$ and $z$ has degree at least 5, then $y$ and $z$ together send charge at least $1+{10\over 7}$ to $t$ by the vertex rules. 
Thus, $\mu^*(t)\geq -3+1+{10\over 7}+3\cdot{1\over 7}+{1\over 7} \geq 0$.

We now assume $d(y) = d(z) = 4$.
Let $s_0, s_1, s_2, \ldots $ be the faces adjacent to $f$ in clockwise order so that $s_1=t$.
Observe $y$ is incident to $s_0$ and $s_1$, while $z$ is incident to $s_1$ and $s_2$.

If $d(s_2)\geq 7$, then $z$ sends charge $2$ to $t$ by (R1) and $y$ sends charge at least $1$ to $t$  by the vertex rules, so $\mu^*(t)\geq -3+2+1=0$. 
Hence $d(s_2) \leq 4$ and by symmetry $d(s_1) \leq 4$.

If $d(s_2)=4$, then since $G$ contains no (C\ref{conf:squareplusc3}), $s_2$ must be incident to at least two high vertices.
Thus, $z$ sends charge ${10\over 7}$ to $t$ by (R2B), and $y$ sends charge at least $1$ to $t$ by the vertex rules. 
Therefore,
$\mu^*(t)\geq -3+{10\over 7}+1+3\cdot{1\over 7}+\frac{1}{7} \geq 0$. 
Hence $d(s_2) =3 $ and by symmetry, also $d(s_0) = 3$.

Since neither the pair $s_0$ and $s_1$, nor the pair $s_1$ and $s_2$ form a copy of (C\ref{conf:c3c3}),
the faces $s_0$ and $s_2$ are good 3-faces.
Let $V(s_0) = \{u,v,y\}$ so that $vy$ is the edge between $s_0$ and $f$. 
See Figure~\ref{fig-discharginghelp_c3}, for an example of this situation.

\begin{figure}[tp]
\begin{center}
\includegraphics{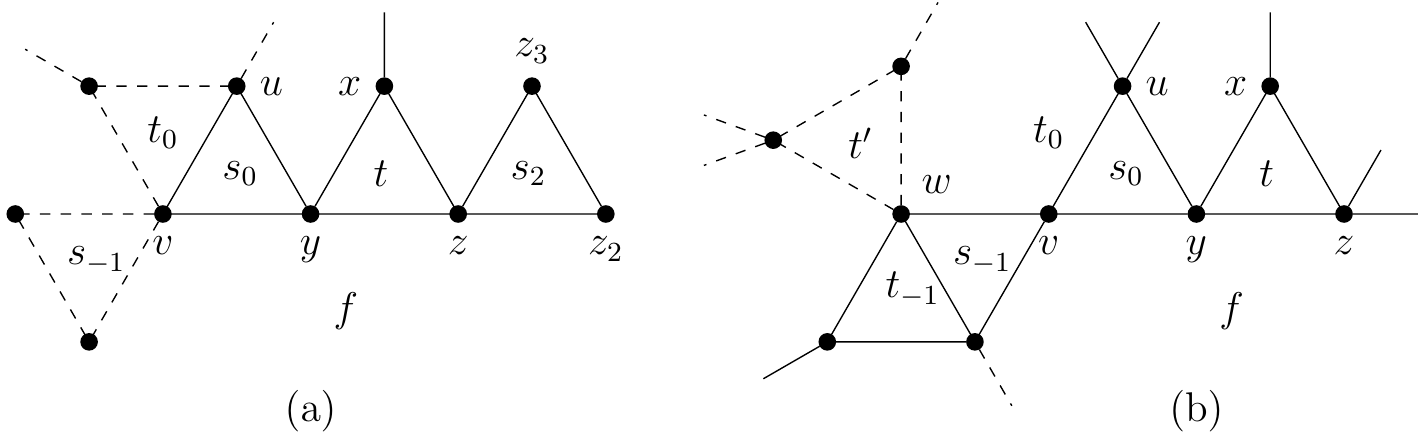}
\caption{Claim~\ref{claim:isotri} where $d(s_0) = d(s_2) = 3$.}
  \label{fig-discharginghelp_c3}
\end{center}
\end{figure}

Suppose that $s_0$ is hungry after vertex rules. Since $s_0$ is not a bad face, it must be in
a diamond with a bad face $t_0$. Since (C\ref{conf:diamondc3c}) is reducible, at least
one of $u,v$ has degree at least 5. If both $u$ and $v$
have degree at least 5 or one has degree at least 6, 
then $t_0$ and $s_0$ receive enough charge after the vertex rules and $s_0$ is not hungry.
Hence without loss of generality, assume $d(v) = 5$ and $d(u)=4$.
If $v$ is incident to another bad face $s_{-1}$, then $s_0,s_{-1},t_{0},s_{t}$ form a reducible configuration (C\ref{conf:diamondNEW2c3s}). 
Hence $s_{-1}$ is not a bad 3-face, so (R3C) applies and the faces $s_0$ and $t_0$ are not hungry.

By a symmetric argument, $s_2$ is also not hungry after the vertex rules.

Recall that $d(x) = 3$, $d(y) = d(z) = 4$, and $s_0$ and $s_2$ are good 3-faces.
Thus, $y$ and $z$ send charge $1$ to $t$ by (R1).
By (R6), the three faces adjacent to $t$ gives charge ${3\over 7}$ to $t$.
If $d(v)\geq 5$, then (R7) applies to $f$ and $f$ contributes $\frac{1}{7}$ to $t$. 
Hence $d(v) = 4$. 
If $s_{-1}$ is not hungry, then again (R7) is applied and $f$ contributes $\frac{1}{7}$.
Hence $s_{-1}$ is hungry. Therefore, $s_{-1}$ is a 3-face and $s_{0}$ is an isolated 3-face.
If $s_{-1}$ is a bad triangle, then $s_{-1},s_{0},t$ form (C\ref{conf:threec3s}) which is reducible.
Hence $s_{-1}$ is not bad and it forms a diamond with a bad 3-face $t_{-1}$.
See Figure~\ref{fig-discharginghelp_c3}(b).
If both vertices shared by $s_{-1}$ and $t_{-1}$ have degree four, faces
$s_{-1},t_{-1},s_{0},t$ form configuration (C\ref{conf:diamondNEWc3chain}). 
If both shared vertices have degree at least 5, then the diamond has nonnegative charge after the vertex rules so $s_{-1}$ cannot be hungry.

Hence one shared vertex $w$ has degree 5 and the other is of degree four.
If $w$ is not incident to any other bad 3-face, then the faces $s_{-1}$ and $t_{-1}$ are not hungry. 
If $w$ is in a bad 3-face $t'$, then $t'$, $s_{-1}$, $t_{-1}$, $s_{0}$, and $t$ form a copy of (C\ref{conf:diamondNEWc3chainPLUS}).
Therefore, (R7) applies and contributes $\frac{1}{7}$.

Thus $\mu^*(t) \geq -3 + 1 + 1 + 3\cdot\frac{1}{7} + 3\cdot\frac{1}{7} + \frac{1}{7} \geq 0$.
\end{proof}


\subsection{Reducibility}\label{subsec:reduce}

We now show that any minimum counterexample $(G,L)$ to Theorem~\ref{thm:c5c6_up} is prime.
Since we already proved that no pair $(G,L)$ is prime, this shows that no counterexample exists.

We start by proving basic properties of $G$.
If $G$ is not connected, there exists an $L$-coloring for every connected
component of $G$ which together give an $L$-coloring of $G$.
Hence $G$ is connected.

Suppose that $G$ has a cut-vertex $v$. Let $G_1$ and $G_2$
be proper subgraphs of $G$ such that $G_1 \cap G_2 = v$, 
$G = G_1 \cup G_2$ and $p \in G_1$. By the minimality of $G$,
there exists an $L$-coloring $\varphi$ of $G_1$. 
Let $L'$ be lists on $G_2$ where $L'(v) = \varphi(v)$ and
$L'(u) = L(u)$ for $u \in V(G_2)-v$. By the minimality of $G$,
there exists an $L'$-coloring $\psi$ of $G_2$. Colorings
$\varphi$ and $\psi$ together give an $L$-coloring of $G$, a contradiction.
Hence $G$ is 2-connected.

Suppose $v\in V(G)- p$ has degree at most two.
An $L$-coloring $\varphi$ of $G - v$ can be extended to $v$ since $|L(v)| \geq 3$.
Hence $d(v) \geq 3$ for every $v \in V(G)-p$.

Suppose $d(p)=1$. Let $v$ be the neighbor of $p$.
Since $G$ is 2-connected, $v$ is not a cut-vertex.
So $V(G)=\{p,v\}$ and $G$ is $L$-colorable.
Hence $d(p) \geq 2$.

By the minimality of $G$, lists of endpoints of every edge $e \in E(G)$
have a color in common. If not, $e$ can be removed from $G$ without
changing possible $L$-colorings. We denote the color shared
by the endpoints of $e$ by $c(e)$.

\begin{lemma}\label{lem:vcolors}
There is no vertex $v$ with a color $c \in L(v)$ not appearing on the edges incident to $v$.
\end{lemma}
\begin{proof}
Suppose $v \in V(G)$ has a color $c \in L(v)$ not appearing in the lists of the adjacent vertices. Let $\varphi$ be an $L$-coloring of $G-v$.
An $L$-coloring of $G$ can be obtained by assigning assigning $\varphi(v) = c$.
\end{proof}

\begin{lemma}\label{lem:trail}
$G$ does not contain a trail of three edges $e_1e_2e_3$ where $c(e_1) = c(e_3) \neq c(e_2)$.
\end{lemma}
\begin{proof}
Suppose $G$ contains a trail of three edges $e_1e_2e_3$ where $c(e_1) = c(e_3) \neq c(e_2)$. 
The lists of the two endpoints of $e_2$ both contain the colors $c(e_1)$ and $c(e_2)$, 
which is a contradiction to the $L$-list assignment if $c(e_1) \neq c(e_2)$. 
\end{proof}

\begin{lemma}\label{lem:lowtriangle}
$G$ does not contain a 3-face $e_1e_2e_3$ incident to a low vertex with $c(e_1) = c(e_2)$.
\end{lemma}
\begin{proof}
If $v$ is a low vertex in a 3-face bounded by $e_1e_2e_3$, then by Lemma~\ref{lem:vcolors} the edges incident to $v$ have distinct colors.
Thus, if $c(e_1)=c(e_2)$, they are not both incident to $v$ and $c(e_3) \neq c(e_1) = c(e_2)$, and thus a trail from Lemma~\ref{lem:trail} is contained in $G$.
\end{proof}

\begin{lemma}\label{lma:lowface}
$G$ contains no copy of (C\ref{conf:pC3}). 
\end{lemma}	
\begin{proof}
Suppose $puv$ is a 3-face. 
Let $c = L(p) = c(pu) = c(pv)$. 
By Lemma~\ref{lem:trail}, also $c(uv)=c$. 
Let $\varphi$ be and $L$-coloring of $G-uv$. 
Since $\varphi(p) = c$, $\varphi(u) \neq c$ and $\varphi(v) \neq c$.
Hence $\varphi(u) \neq \varphi(v)$ and $\varphi$ is an $L$-coloring of $G$.
\end{proof}

We will show that a minimum counterexample $(G,L)$ contains no copy of the configurations (C\ref{conf:lowface})--(C\ref{conf:last}) by using a concrete form of reducibility.
\begin{definition}
A configuration $C$ is \emph{reducible} if there exist disjoint sets $X, R \subseteq V(C)$ where $X$ is nonempty, $p \notin X\cup R$, the set $X \cup R$ contains exactly the vertices of $C$ with at most one neighbor outside $C$, and for every $L$-coloring $\varphi$ of $G - X$, there exists an $L$-coloring $\psi$ of $G$ satisfying the following properties:
\begin{itemize}
	\item $\varphi(v) = \psi(v)$ for all $v \notin X \cup R$, 
	\item if $\varphi(x) \neq \psi(x)$ for some $x \in R$ with one neighbor outside of $C$, then $|L(x) \cap \{ \psi(y) : y \in N(x) \cap V(C) \}| \leq 1$, and
	\item if $\varphi(x) \neq \psi(x)$ for some $x \in R$ with no neighbor outside of $C$, then $|L(x) \cap \{ \psi(y) : y \in N(x) \cap V(C) \}| \leq 2$.
\end{itemize}
\end{definition}
If a graph $G$ contains a copy of a reducible configuration, then it is not a minimum counterexample since $L$-colorings of proper subgraphs extend to $L$-colorings of $G$.
We now use this definition to prove our simple configurations (C\ref{conf:lowface})--(C\ref{conf:lastsimple}) are reducible.

\begin{lemma}\label{lma:lowface}
(C\ref{conf:lowface}) is reducible. 
\end{lemma}	
\begin{proof}
Let $F=v_1v_2v_3v_4$ be a normal 4-face where $d(v_i) = 3$ for $i \in \{1,2,3,4\}$; see Figure~\ref{fig-reducible-noted}.
Let $X = F$ and $R = \varnothing$ and let $\varphi$ be an $L$-coloring of $G-X$.
Each vertex $v_i$ has at most one color forbidden by $\varphi$, which implies that $|L_\varphi(v_i)|\geq 2$.
The only case when a cycle is not $2$-choosable is when the cycle has odd length and each vertex has the same list of size $2$. 
Hence $\varphi$ can be extended to an $L$-coloring $\psi$ of $G$.
\end{proof}

\begin{figure}[tp]
\begin{center}
\includegraphics{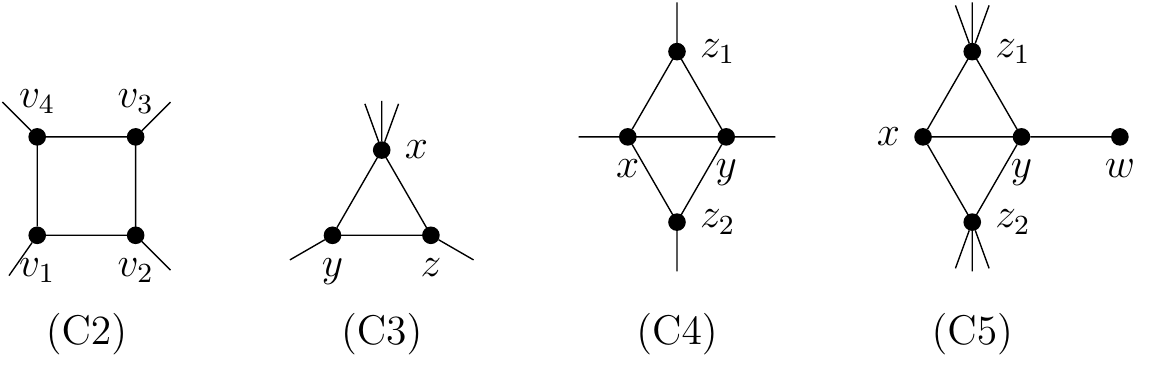}
\caption{Situations in Lemmas~\ref{lma:lowface}, \ref{lem:c3-1high}, \ref{lem:baddiamond}, and \ref{lem:crown}.}
  \label{fig-reducible-noted}
\end{center}
\end{figure}

\begin{lemma}\label{lem:c3-1high}
(C\ref{conf:c3-1high}) is reducible. 
\end{lemma}	
\begin{proof}
Let $xyz$ be a 3-face where $y$ and $z$ have degree 3; see Figure~\ref{fig-reducible-noted}.
Let $X = \{y,z\}$ and $R = \varnothing$, and let $\varphi$ be an $L$-coloring of $G-X$.
By Lemmas~\ref{lem:vcolors} or \ref{lem:trail}, the color $\varphi(x)$ is not equal to both $c(xy)$ and $c(xz)$.
Thus, without loss of generality, we assume $\varphi(x)\neq c(xz)$.
Observe $|L_\varphi(y)| \geq 1$ and $|L_\varphi(z)| \geq 2$, and thus we can color $y$ and then $z$ to find an $L$-coloring $\psi$ of $G$.
\end{proof}

\begin{lemma}\label{lem:baddiamond}
(C\ref{conf:baddiamond}) is reducible. 
\end{lemma}	
\begin{proof}
Let $z_1xyz_2$ be a diamond as in (C\ref{conf:baddiamond}); see Figure~\ref{fig-reducible-noted}.
Let $X = \{z_1, x, y, z_2\}$ and $R = \varnothing$, and let $\varphi$ be an $L$-coloring of $G-X$.
Observe $|L_\varphi(v)| \geq 2$ for all $v \in \{z_1,x,y,z_2\}$.
If the color $c(xy)$ no longer appears in both $L_\varphi(x)$ and $L_\varphi(y)$, then we can remove the edge $xy$ and extend the coloring to an $L$-coloring $\psi$ of $G$ since $z_1xz_2y$ is a 4-cycle, which is 2-choosable.

Suppose that $c(xy) \in L_\varphi(x)$.
By Lemma~\ref{lem:lowtriangle}, $c(xy) \neq c(xz_1)$ and  $c(xy) \neq c(xz_2)$.
Set $\varphi(x) = c(xy)$. Now $\varphi$ can be extended to an $L$-coloring of $G$
by coloring $y$ and then $z_1,z_2$ in a greedy way.
\end{proof}

\begin{lemma}\label{lem:crown}
(C\ref{conf:crown}) is reducible. 
\end{lemma}	
\begin{proof}
Let $z_1xyz_2$ be a diamond as in (C\ref{conf:crown}); see Figure~\ref{fig-reducible-noted}.
Let $X = \{x\}$ and $R = \{y\}$, and let $\varphi$ be an $L$-coloring of $G - X$.
By Lemma~\ref{lem:vcolors}, the colors $c(xv)$ are distinct for $v \in \{z_1, y, z_2\}$.
If $\varphi(v) \neq c(xv)$ for some $v\in \{z_1, y, z_2\}$, then we color $\varphi(x) = c(xv)$ to find an $L$-coloring of $G$ without recoloring $y \in R$.
Thus, $\varphi(v) = c(xv)$ for all $v \in \{z_1,y,z_2\}$.

By Lemma~\ref{lem:trail}, the color $c(z_ix)$ is distinct from $c(z_iy)$ for each $i \in \{1,2\}$.
Thus, the colors $\varphi(z_i)$ are not in $L(y)$, so $|L(y) \cap \{ \varphi(v) : v \in \{x,z_1,z_2\}\}| = 1$.
Thus there is at least one color $a \in L(y)$ other than $c(xy)$ and $c(yw)$, where $w$ is the neighbor of $y$ outside the diamond.
We color $\psi(x) = c(xy)$ and recolor $\psi(y) = a$ to find an $L$-coloring of $G$.
\end{proof}

\begin{lemma}
(C\ref{conf:facialk4}) is reducible. 
\end{lemma}
\begin{proof}
Observe that (C\ref{conf:facialk4}) contains (C\ref{conf:crown}) as a subgraph
and the proof for (C\ref{conf:crown}) works also for (C\ref{conf:facialk4}).
\end{proof}	

To complete the list of reducible configurations, we describe a way to build compound
reducible configurations from simple reducible configurations by adding
a bad face.

\begin{lemma}[Iterative Construction]\label{glue}
Let $C$ be a reducible configuration and let $v\in V(C)$ have a unique neighbor $u \in N(v) \setminus V(C)$.
Let $C'$ be obtained from $C$ by removing the edge $vu$ and adding two
new vertices $x,y$ such that $vxy$ is a 3-face, $y$ has exactly one neighbor $z$ in $N(y)\setminus V(C')$ and $x$ has at least one neighbor in $N(x)\setminus V(C')$.
If $C$ is reducible, then $C'$ is reducible.
\end{lemma}
See Figure~\ref{fig-rules-generating-noted} for a visualization of this construction.
\begin{proof}
Let $G$ contain $C'$.
Observe that since $y$ has degree three, Lemmas~\ref{lem:vcolors} and \ref{lem:trail} guarantee that $c(vy)$, $c(yx)$, and $c(vx)$ are distinct.

Let $X, R \subseteq V(C)$ be given by the definition of $C$ being reducible.
Let $X' = X$ and $R' = R \cup \{y \}$.
We consider cases based on whether $v$ is in $X$ or $R$.

\begin{figure}[tp]
\begin{center}
\includegraphics{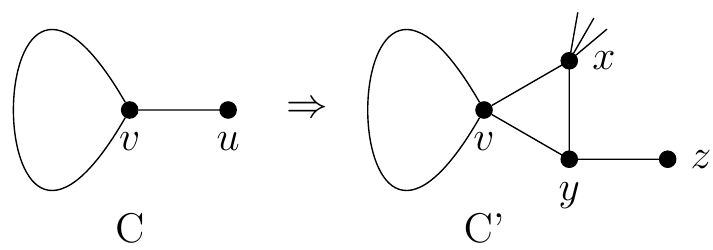}
\caption{Creating compound reducible configurations in Lemma~\ref{glue}.}
  \label{fig-rules-generating-noted}
\end{center}
\end{figure}

\begin{mycases}
\mycase{$v \in X$}
For an $L$-coloring $\varphi$ of $G - X$, we use the method of extending a coloring to the vertices in $C$ given by its proof of reducibility.
When coloring $v$, the method expects only one color from $L(v)$ appearing in its neighbors outside of $C$.
If at most one of $\varphi(x)$ or $\varphi(y)$ appears in $L(v)$, then the method to color $C$ completes with an $L$-coloring $\psi$ of $G$.
Otherwise, $\varphi(y) = c(vy)$ and $\varphi(x) = c(vx) \neq c(yx)$.
Thus, we recolor $\psi(y) = c(yx)$ and assign $\psi(v) = \varphi(y)$.
This recolors $y$ with a color that does not appear in its neighbors, and $y$ has exactly one color restricted within $C'$.

\mycase{$v \in R$}
For an $L$-coloring $\varphi$ of $G - X'$, we use the method of extending a coloring to the vertices in $C$ given by its proof of reducibility.
If it can be colored without recoloring $v$, then the resulting coloring is an $L$-coloring on $G$.
However, if $v$ must be recolored, then $v$ has at most one color restricted from within $C$.
Since $c(vx) \neq c(vy)$, if $v$ has no available colors for this recoloring, we have $\varphi(y) = c(vy)$ and $\varphi(x) = c(vx) \neq c(xy)$.
Thus, we recolor $\psi(v) = \varphi(y) = c(vy)$ and $\psi(y) = c(yx)$.
Observe that $v$ has at most two colors restricted by its neighbors in $C'$, and $y$ has at most one color restricted by its neighbors in $C'$.
\end{mycases}
In either case, we have modified the coloring algorithm for $C$ to apply for $C'$.
\end{proof}

\begin{lemma}
(C\ref{conf:firstcompound})--(C\ref{conf:last}) are  reducible. 
\end{lemma}
\begin{proof}
Each configuration is built using Lemma~\ref{glue} from a known reducible configuration.
We use the notation ``(C$i$) $\stackrel{\text{ }}{\longrightarrow}$ (C$j$)'' to denote ``Applying Lemma~\ref{glue} to (C$i$) results in (C$j$),'' in the following pairs:
\begin{center}
\begin{tabular}[h]{rcl@{\qquad}rcl@{\qquad}rcl}
(C\ref{conf:facialk4}) 
&$\stackrel{\text{ }}{\longrightarrow}$ 
&(C\ref{conf:facialk4c3})
&

(C\ref{conf:baddiamond})
&$\stackrel{\text{ }}{\longrightarrow}$ 
&(C\ref{conf:diamondc3a}) 
&

(C\ref{conf:crown})
&$\stackrel{\text{ }}{\longrightarrow}$ 
&
(C\ref{conf:diamondc3b}) 
\\

(C\ref{conf:baddiamond})
&$\stackrel{\text{ }}{\longrightarrow}$ 
&(C\ref{conf:diamondc3c}) 
&

(C\ref{conf:lowface})
&$\stackrel{\text{ }}{\longrightarrow}$ 
&(C\ref{conf:squareplusc3}) 

&
(C\ref{conf:c3-1high})
&$\stackrel{\text{ }}{\longrightarrow}$ 
&(C\ref{conf:c3c3}) 

\\
(C\ref{conf:c3c3})
&$\stackrel{\text{ }}{\longrightarrow}$ 
&(C\ref{conf:threec3s}) 
&
(C\ref{conf:diamondc3a})
&$\stackrel{\text{ }}{\longrightarrow}$ 
&(C\ref{conf:diamondNEW2c3s}) 
&
(C\ref{conf:diamondc3a})
&$\stackrel{\text{ }}{\longrightarrow}$ 
&(C\ref{conf:diamondNEWc3chain}) 
\\
&&&
(C\ref{conf:diamondNEWc3chain})
&$\stackrel{\text{ }}{\longrightarrow}$ 
&(C\ref{conf:diamondNEWc3chainPLUS}) 

\end{tabular}
\end{center}
Thus, by Lemma~\ref{glue} and previous lemmas, these configurations are reducible.
\end{proof}

\section{Conclusion}

The main problem if planar graphs are $(3,1)$-choosable remains open. 
We hope that this paper could serve as an inspiration of possible approaches to the problem.
Unfortunately, the conditions of Theorem~\ref{thm:thomassen3cycle} and Theorem~\ref{thm:thomassen4cycle} are not valid for all planar graphs; see Figure~\ref{fig-notgeneralize}.
Let us note that we do not have an example where $P$ has length one and the endpoints are not in a triangle.

\begin{figure}[htp]
\begin{center}
\includegraphics{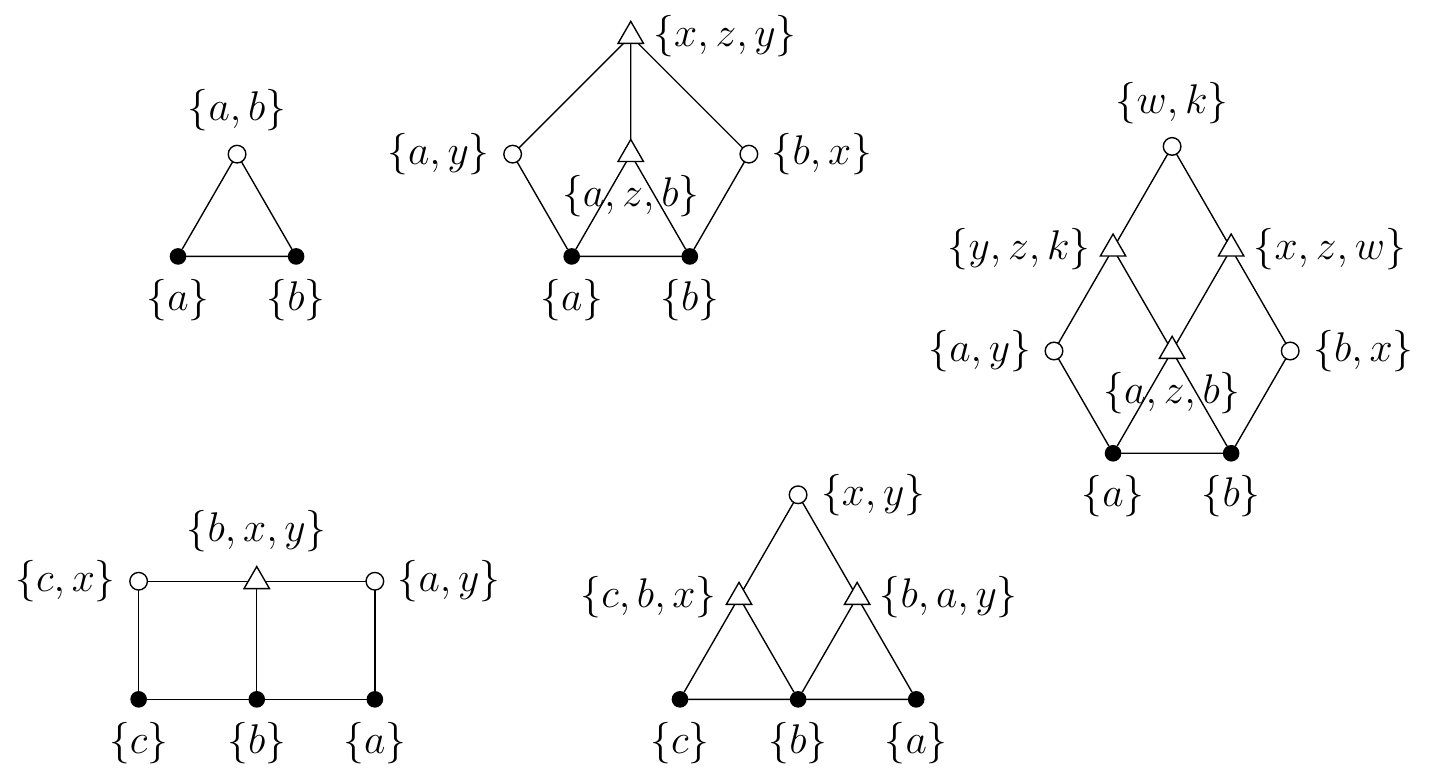}
\caption{Some examples where conditions of Theorem~\ref{thm:thomassen3cycle} and Theorem~\ref{thm:thomassen4cycle} do not generalize to all planar graphs.}
  \label{fig-notgeneralize}
\end{center}
\end{figure}

The last thing we promised is a planar graph $G$ without 4-cycles and 5-cycles that is not $(3,2)$-choosable.
It is a modification of a construction of Wang, Wen, and Wang~\cite{08WaWeWa}.
The main building gadget is the graph $H$ depicted in Figure~\ref{fig-not32}. It has two vertices with lists
of size one. The graph $G$ is created by taking 9 copies of $H$ and identifying vertices with lists $\{a\}$
into one vertex $v$ and vertices with lists $\{b\}$ into one. 
Vertices $u$ and $v$ get disjoint lists and we assign to every 9 possible colorings of $u$ and $v$ one
gadget, where the coloring of $u$ and $v$ cannot be extended. By inspecting the gadget, the reader
can check that $G$ cannot be colored and that $G$ has no cycles of length 4 or 5.

\begin{figure}[htp]
\begin{center}
\includegraphics[scale=0.8]{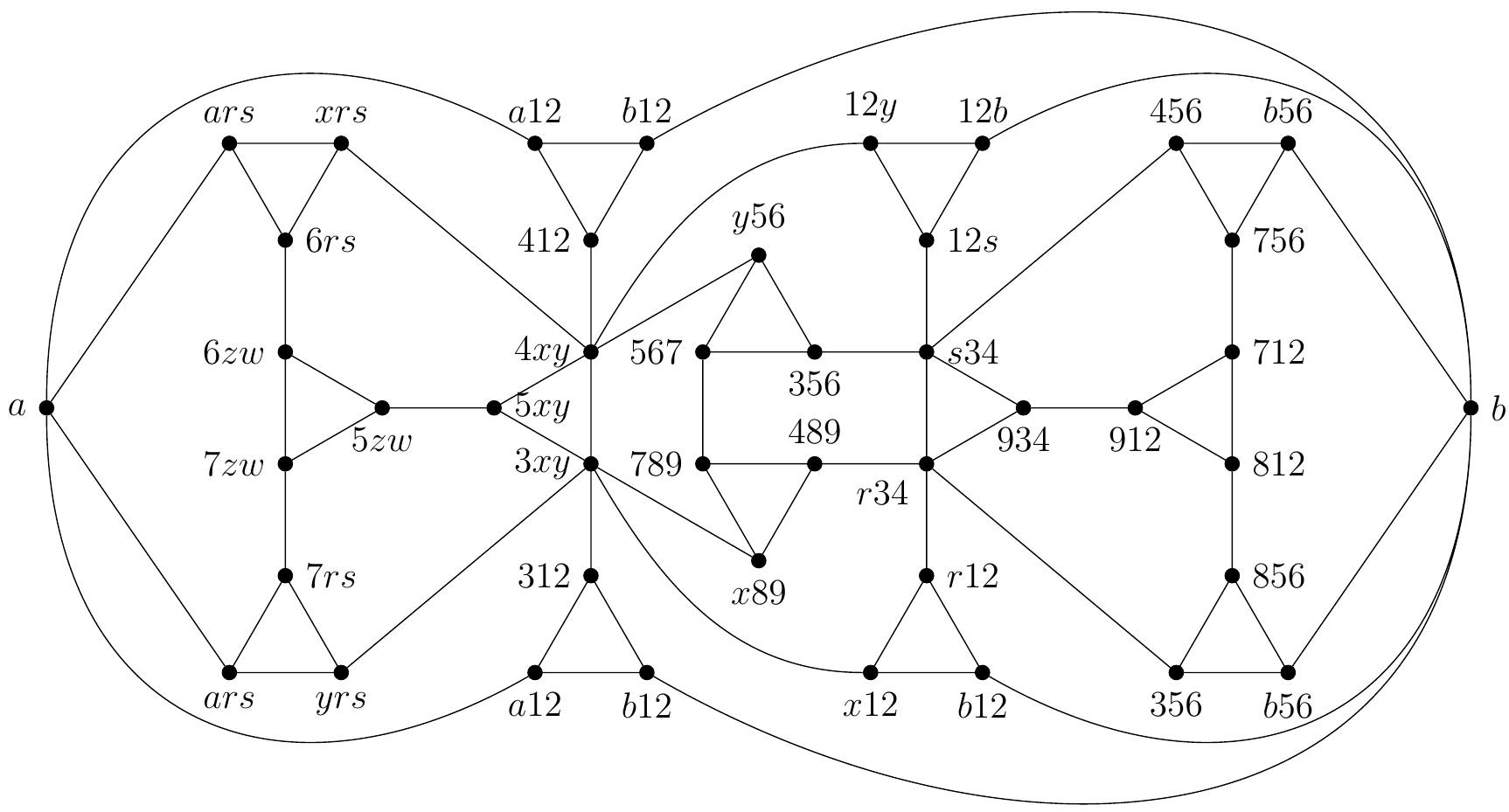}
\caption{Building block of a non $(3,2)$-choosable planar graph without cycles of length $4$ and $5$.}
  \label{fig-not32}
\end{center}
\end{figure}

The authors thank Mohit Kumbhat for introducing them to the problem 
during $3^{\text{rd}}$ Eml\'{e}kt\'{a}bla Workhops and thank Kyle F. Jao for fruitful discussions and encouragement in the early stage of the project.


\end{document}